\documentclass{amsart}
\usepackage{amsmath,amssymb,amscd,latexsym}

\usepackage[all]{xy}

\usepackage[usenames, dvipsnames]{color}

\newcommand{\A}{{\mathbb{A}}}
\newcommand{\C}{{\mathbb{C}}}
\newcommand{\DD}{\Delta}

\newcommand{\ED}{E_{\Dh}}
\newcommand{\Elog}{E_{\log}}
\newcommand{\EDB}{E_{\log}}
\newcommand{\FD}{F_{\Dh}}

\newcommand{\MUD}{MU_{\Dh}}
\newcommand{\MUlog}{MU_{\log}}
\newcommand{\MUDB}{MU_{\log}}

\newcommand{\Lee}{\mathbb{L}}
\newcommand{\Q}{{\mathbb{Q}}}
\newcommand{\Pro}{{\mathbb{P}}}
\newcommand{\Sph}{{\mathbb{S}}}
\newcommand{\R}{{\mathbb{R}}}
\newcommand{\Z}{{\mathbb{Z}}}
\newcommand{\an}{\mathrm{an}}

\newcommand{\clH}{\mathrm{cl}_{H}}
\newcommand{\clHD}{\mathrm{cl}_{H_{\Dh}}}
\newcommand{\clMU}{\mathrm{cl}_{MU}}
\newcommand{\clMUD}{\mathrm{cl}_{\MUlog}}

\newcommand{\Ev}{\mathrm{Ev}}
\newcommand{\Hdg}{\mathrm{Hdg}}
\newcommand{\HdgE}{\mathrm{Hdg}_{E}}
\newcommand{\HdgMU}{\mathrm{Hdg}_{MU}}

\newcommand{\inj}{\mathrm{inj}}

\newcommand{\nomo}{\mathrm{nomo}}

\newcommand{\proj}{\mathrm{proj}}
\newcommand{\pt}{\mathrm{pt}}

\newcommand{\sing}{\mathrm{sing}}
\newcommand{\Sing}{\mathrm{Sing}_*}

\newcommand{\Spec}{\mathrm{Spec}\,}

\newcommand{\colim}{\operatorname*{colim}}
\newcommand{\hocolim}{\operatorname*{hocolim}}
\newcommand{\holim}{\operatorname*{holim}}
\newcommand{\Hom}{\mathrm{Hom}}

\newcommand{\Homhs}{{\mathcal Hom}_{\Sph}}

\newcommand{\Imm}{\mathrm{Im}\,}
\newcommand{\Ker}{\mathrm{Ker}\,}

\newcommand{\Map}{\mathrm{Map}}

\newcommand{\Sm}{\mathbf{Sm}}
\newcommand{\Smc}{\Sm_{\C}}
\newcommand{\Smb}{\overline{\Smc}}
\newcommand{\Smcnis}{\Sm_{\C,\mathrm{Nis}}}

\newcommand{\Manc}{\mathbf{Man}_{\C}}

\newcommand{\Pre}{\mathbf{Pre}}
\newcommand{\sPre}{\mathbf{sPre}}
\newcommand{\sPrep}{\mathbf{sPre}_*}
\newcommand{\hosPre}{\mathrm{ho}\sPre}

\newcommand{\SpsPre}{\sSp(\sPrep)}
\newcommand{\hoSpsPre}{\mathrm{ho}\sSp(\sPrep)}
\newcommand{\SpsPreman}{\sSp(\sPrep(\Manc))}
\newcommand{\hoSpsPreman}{\mathrm{ho}\SpsPreman}
\newcommand{\SpsPrenis}{\sSp(\sPrep(\Smcnis))}
\newcommand{\hoSpsPrenis}{\mathrm{ho}\SpsPrenis}
\newcommand{\HCAlg}{H\C-\mathbf{Alg}}

\newcommand{\Predga}{\mathbf{PreDGA}_{\C}}

\newcommand{\sS}{\mathbf{sS}}
\newcommand{\hosS}{\mathrm{ho}\sS}
\newcommand{\sSp}{\mathrm{Sp}^{\Sigma}}

\newcommand{\sSpt}{\mathrm{Sp}_{\mathrm{Top}}^{\Sigma}}
\newcommand{\hosSpt}{\mathrm{ho}\sSpt}
\newcommand{\Top}{\mathbf{Top}}
\newcommand{\hoTop}{\mathrm{ho}\Top}
\newcommand{\Ch}{{\mathcal C}}
\newcommand{\Dh}{{\mathcal D}}
\newcommand{\Eh}{{\mathcal E}}

\newcommand{\Fh}{{\mathcal F}}

\newcommand{\Gh}{{\mathcal G}}

\newcommand{\Oh}{{\mathcal O}}

\newcommand{\Sh}{{\mathcal S}}

\newcommand{\Uh}{{\mathcal U}}
\newcommand{\Vh}{{\mathcal V}}
\newcommand{\Wh}{\mathcal{W}}
\newcommand{\Xh}{\mathcal{X}}

\newcommand{\Yh}{\mathcal{Y}}
\newcommand{\Zh}{\mathcal{Z}}
\newcommand{\oX}{\overline{X}}
\newcommand{\oY}{\overline{Y}}

\newcommand{\into}{\hookrightarrow}

\newtheorem{theorem}{Theorem}[section]
\newtheorem{lemma}[theorem]{Lemma}
\newtheorem{prop}[theorem]{Proposition}
\newtheorem{cor}[theorem]{Corollary}
\theoremstyle{definition}
\newtheorem{defn}[theorem]{Definition}

\newtheorem{example}[theorem]{Example}

\newtheorem{remark}[theorem]{Remark}
\begin{document}
\title{Hodge filtered complex bordism}
\author{Michael J. Hopkins}
\address{Department of Mathematics, Harvard University, Cambridge, MA 02138, USA}
\email{mjh@math.harvard.edu}
\author{Gereon Quick}
\thanks{Both authors were supported in part by DARPA under HR0011-10-1-0054-DOD35CAP. The first author was supported in part by the National Science Foundation under DMS-0906194. The second author was supported by  the German Research Foundation (DFG)-Fellowship QU 317/1}
\address{Mathematisches Institut, WWU M\"unster, Einsteinstr. 62, 48149 M\"unster, Germany}
\email{gquick@math.uni-muenster.de}
\date{}
%\maketitle
\begin{abstract}
We construct Hodge filtered cohomology groups for complex manifolds that combine the topological information of generalized cohomology theories with geometric data of Hodge filtered holomorphic forms. This theory provides a natural generalization of Deligne cohomology. For smooth complex algebraic varieties, we show that the theory satisfies a projective bundle formula and $\A^1$-homotopy invariance. Moreover, we obtain transfer maps along projective morphisms. %Our main motivation is a possible refinement of the Abel-Jacobi map for smooth projective complex varieties.
\end{abstract}

\maketitle

\section{Introduction}
%
%A complex manifold $X$ is a topological object that is equipped with a rich geometric structure. For example the splitting of the tangent space of $X$ under the complex structure induces a decomposition of the bundle of complex differential forms. This decomposition determines a Hodge filtration on the cohomology. If $X$ is in addition a compact K\"ahler manifold, which means that the differential form associated to its hermitian metric is closed, then $H^k(X;\C)$ carries a Hodge structure, i.e. the splitting of the bundle of differential forms induces a complete decomposition of the cohomology into a direct sum $H^{p,q}(X;\C)$ of subspaces which satisfy certain properties that allow a deep understanding of the cohomology of $X$. 
%
For a complex manifold, Deligne cohomology is an elegant way to combine the topological information of integral cohomology with the geometric data of holomorphic forms. For a given $p$, the Deligne cohomology group $H_{\Dh}^n(X;\Z(p))$ of a complex manifold $X$ is defined to be the $n$th hypercohomology of the Deligne complex $\Z_{\Dh}(p)$ 
\[
0 \to \Z \xrightarrow{(2\pi i)^p} \Oh_X \to \Omega^1_X \xrightarrow{d} \ldots \xrightarrow{d} \Omega^{p-1}_X \to 0
\]
of sheaves over $X$, where $\Omega_X^k$ denotes the sheaf of holomorphic $k$-forms. 

Deligne cohomology is an important tool for many fundamental questions in algebraic, arithmetic and complex geometry. Let us just highlight one of the algebraic geometric applications of Deligne cohomology. Let $X$ be a smooth projective complex variety. The associated complex manifold is an example of a compact K\"ahler manifold. Let $CH^pX$ be the Chow group of codimension $p$ algebraic cycles modulo rational equivalence. There is the classical cycle map 
\[
\clH \colon CH^pX \to H^{2p}(X;\Z)
\]
from Chow groups to the integral cohomology of the associated complex manifold of $X$. This map sends an irreducible subvariety $Z\subset X$ of codimension $p$ in $X$ to the Poincar\'e dual of the fundamental class of a resolution of singularities of $Z$. The so defined cohomology class of $Z$ is an integral Hodge class, i.e. lies in the subgroup 
\[
\Hdg^{2p}(X) \subset H^{2p}(X;\Z)
\]
of integral cohomology classes whose images in $H^{2p}(X;\C)$ lie in the step $F^pH^{2p}(X; \C)$ 
of the Hodge filtration of the de Rham cohomology. 

Now Deligne cohomology groups enter the picture. Deligne showed that for every given $p$ the group $H_{\Dh}^{2p}(X;\Z(p))$ sits in a short exact sequence
\begin{equation}\label{sesDeligneintro}
0\to J^{2p-1}(X) \to H_{\Dh}^{2p}(X; \Z(p)) \to \Hdg^{2p}(X)\to 0
\end{equation}
where the group on the left hand side is a compact complex torus, the $p$th intermediate Jacobian of $X$. 
Moreover, Deligne constructed a cycle map 
\[
\clHD \colon CH^pX \to H_{\Dh}^{2p}(X;\Z(p)).
\]
Together with the short exact sequence (\ref{sesDeligneintro}) this provided an algebraic definition of Griffiths's Abel-Jacobi homomorphism from cycles homologically equivalent to zero, i.e. those which are mapped to zero under $\clH$, to the intermediate Jacobian $J^{2p-1}(X)$ (see \cite{zeinzucker}). Thus Deligne groups play an important role in the study of the cycle map, the Abel-Jacobi map and, in particular, the Griffiths group of smooth projective complex varieties, i.e. the quotient group of cycles homologically equivalent to zero modulo a potentially weaker relation, called algebraic equivalence. 

More recently Totaro showed that the Griffiths group of a complex variety can be analyzed using the purely topological information of its complex cobordism groups. In \cite{totaro} he proved that the map $\clH$ factors through a cycle map 
\[
\clMU \colon CH^pX \to (MU^{*}(X)\otimes_{MU^*}\Z)^{2p}
\]
for any smooth projective complex variety $X$; in fact, Totaro showed
a more general result for any complex variety in terms of the
corresponding bordism quotient groups.  Since the canonical map
$\theta \colon MU^{2p}(X)\otimes_{MU^*}\Z \to H^{2p}(X;\Z)$ is an
isomorphism after tensoring with $\Q$, this factorization is a torsion
phenomenon. But in general, the kernel of $\theta$ can be nontrivial.
Generalizing work of Atiyah and Hirzebruch, Totaro constructed (using
Godeaux-Serre varieties)
elements in the kernel of $\theta$ that lie in the image of $\clMU$.
In this way he found new examples of nontrivial elements in the
Griffiths group of Godeaux-Serre varieties. In \cite{cycles}, the
second author extended Totaro's idea for the cycle map to varieties
over algebraically closed fields of positive characteristic.

Now given that Deligne cohomology and complex cobordism are useful
tools for the understanding of the cycle map and also for the
Abel-Jacobi map, the question arises whether it is possible to combine
these two approaches.  In other words, is it
possible to combine the geometric information of differential forms on
a complex manifold with the topological information of complex
cobordism groups which carry richer data than cohomology groups? The
purpose of this paper is to give a positive answer to this question by
constructing a natural generalization of Deligne cohomology via
complex cobordism and to set the stage for its applications, in
particular with a view towards a new Abel-Jacobi map.

The idea for the construction is similar to the one for generalized differential cohomology theories in \cite{hs}. Let $X$ be a complex manifold. The Deligne complex $\Z_{\Dh}(p)$ is quasi-isomorphic to the homotopy pullback of the diagram of sheaves of complexes 
\[
\xymatrix{
 & \Z \ar[d] \\
 \Omega_X^{*\geq p}\ar[r] & \Omega_X^*}
 \]
where $\Omega^*_X$ denotes the complex of holomorphic forms on $X$ and $\Omega^{*\geq p}_X$ denotes the subcomplex of forms of degree at least $p$. We would like to replace the complex $\Z$ which contributes the singular cohomology of $X$ with a spectrum representing a more general cohomology theory. In particular, we would like to use the Thom spectrum $MU$ of complex bordism.

In order to make this work, we have to throw the latter chain complexes and $MU$ in a common category. We do this by adding a simplicial direction and considering simplicial sets rather than chain complexes. Moreover, in order to have a good homotopy theory for complex manifolds as part of the team we consider presheaves of spectra on the site of complex manifolds with open coverings. In order to illustrate the basic pattern for the construction, we will start in the next section with a reformulation of Deligne cohomology in terms of simplicial presheaves before we move on to the generalizations. 

The technical details for the general case are a bit more involved. In particular, to obtain a nice Eilenberg-MacLane spectrum functor from presheaves of differential graded algebras to presheaves of symmetric spectra which preserves the product structure requires more machinery on the pitch than one would like. The construction is based on the idea of Brown in \cite{brown} to use sheaves of spectra to define generalized sheaf cohomology. 

Once our roster is complete, we can construct for every integer $p$ and any symmetric spectrum $E$ such that $\pi_jE\otimes \C$ vanishes if $j$ is odd, a new presheaf of symmetric spectra $\ED(p)$ as a suitable homotopy pullback and define the Hodge filtered complex bordism groups $\ED^n(p)(X)$ of $X$ as the homotopy groups of $\ED(p)$.  

Given a map of symmetric spectra $E \to H\Z$ from $E$ to the integral Eilenberg-MacLane spectrum, we obtain a natural homomorphism
\begin{equation}\label{MUDtoHDintro}
\ED^n(p)(X) \to H^n_{\Dh}(X;\Z(p))
\end{equation}
between Hodge filtered $E$-cohomology and Deligne cohomology groups of any complex manifold $X$. For $E=MU$, we will show that Hodge filtered complex bordism has a multiplicative structure and that map (\ref{MUDtoHDintro}) respects the ring structures on both sides.

One motivation for this project is the result that for a compact K\"ahler manifold $X$ the diagonal Hodge filtered cohomology groups $\ED^{2p}(p)(X)$ sit in an exact sequence 
\begin{equation}\label{sesMUintro}
0\to J_E^{2p-1}(X) \to \ED^{2p}(p)(X) \to \HdgE^{2p}(X)\to 0.
\end{equation}
The group on the right is the subset of elements in $E^{2p}(X)$ that map to Hodge classes in cohomology. The left hand group is an analogue of the intermediate Jacobian and carries the structure of a complex torus. Moreover, a map $E\to H\Z$ induces a natural map of short exact sequences from \eqref{sesMUintro} to \eqref{sesDeligneintro}. 

We have written the first five sections of this paper in a way that applies to complex manifolds.	In the final section we specialize to the case of smooth algebraic varieties and modify our construction in order to take mixed Hodge structures into account. Following Deligne and Beilinson, we replace the sheaves of holomorphic forms occurring in the definition of $\ED(p)$ with the sheaves of holomorphic forms having logarithmic poles to form $E_{\log}(p)$. These "logarithmic" generalized cohomology groups are finitely generated if the groups $E_\ast (\text {pt})$ are, and are $ \A^1$-invariant. 

In the case that $E$ is the topological Thom spectrum $MU$, we obtain a projective bundle formula and transfer maps for projective morphisms. This transfer structure provides a new cycle map that associates to a smooth irreducible subvariety $Z$ of codimension $p$ of a smooth projective complex variety $X$ an element in $\MUlog^{2p}(p)(X)$. This also induces a generalized Abel-Jacobi homomorphism on smooth cycles that map to zero in $\HdgMU^{2p}(X)$, a subgroup of the cycles homologically equivalent to zero. More generally, for every smooth quasi-projective complex variety $X$, we obtain a natural ring homomorphism from algebraic cobordism of Levine and Morel \cite{lm} to Hodge filtered complex bordism
\[
\Omega^*(X) \to \MUlog^{2*}(*)(X).
\] 

Finally, by \cite{lm}, there is a natural isomorphism 
$CH^* \cong \Omega^*\otimes_{\Lee^*}\Z$ of oriented cohomology theories on $\Smc$. This implies that there is a natural homomorphism
\[
CH^*X \to \MUlog^{2*}(*)(X)\otimes_{MU^*}\Z.
\]

We conclude the paper with two examples of elements in algebraic
cobordism that map to zero in the Chow ring and in complex cobordism
but remain non-zero in $\MUlog{2*}(*)(X)$.

The authors wish to thank Clark Barwick, H\'el\`ene Esnault, Marc
Levine, Burt Totaro, and Claire Voisin for very helpful conversations
and comments. We would also like to thank the anonymous referee for
many very helpful comments and suggestions to improve the paper.

\section{Deligne cohomology in terms of simplicial presheaves}

We would like to modify the definition of Deligne cohomology groups in a way that allows a generalization. The new theory should fit in an exact sequence similar to (\ref{sesDeligneintro}) with the role of singular cohomology replaced by a generalized cohomology theory. 

The starting observation is that we can consider the Deligne complex $\Z_{\Dh}(p)$ also as a homotopy pullback of complexes. Up to quasi-isomorphism $\Z_{\Dh}(p)$ fits into the homotopy pullback square of complexes
\begin{equation}\label{pullbackzdp}
\xymatrix{
\Z_{\Dh}(p) \ar[d] \ar[r] & \Z \ar[d]\\
\Omega_X^{*\geq p} \ar[r] & \Omega_X^*}
\end{equation}
where $\Omega_X^*$ denotes the full complex of holomorphic differential forms and $\Omega_X^{*\geq p}$ is the truncated subcomplex of $\Omega_X^*$ of forms of degree at least $p$. The right hand vertical map in (\ref{pullbackzdp}) is given by $(2\pi i)^p$ times the canonical inclusion $\Z \into \C$. Taking the homotopy pullback along the lower horizontal map corresponds to cutting the complex $0 \to \Z \to \Omega_X^*$ at degree $p$. 

Now we would like to substitute the sheaf of chain complexes $\Z$ by a different player. In order to set  the stage for this replacement we consider diagram (\ref{pullbackzdp}) in the world of simplicial presheaves on complex manifolds. 
%
%%%%%%%%%%%%%%%%%%%%%%%%%%%%%%%%%%%%%%

\subsection{Simplicial presheaves}\label{section2.1}

Let $\Manc$ be the category of complex manifolds and holomorphic maps. We consider it as a site with the Grothendieck topology defined by open coverings. %, or equivalently local homeomorphisms which are holomorphic maps. 

Let $\Pre$ be the category of presheaves of sets on the site $\Manc$. We denote by $\sPre$ the category of simplicial objects in $\Pre$, or equivalently the category of presheaves of simplicial sets on $\Manc$. Sending an object $U$ of $\Manc$ to the presheaf it represents defines a fully faithful embedding $\Manc \into \Pre$. Since any presheaf defines an object in $\sPre$ of simplicial dimension zero, we can also embed $\Manc$ into $\sPre$. On the other hand every simplicial set defines a simplicial presheaf. For example, the simplicial circle $S^1=\Delta[1]/\partial\Delta[1]$ can be considered as a simplicial presheaf by sending any complex manifold $U$ to $S^1$.

%The category $\sPre$ carries a natural structure of a simplicial category. For two simplicial presheaves $\Xh$ and $\Yh$, the simplicial function space $\map_{\sPre}(\Xh,\Yh)$ is given by the simplicial set 
%$$\map_{\sPre}(\Xh,\Yh) = \Hom_{\sPre}(\Xh \times \DD^{\bullet}, \Yh).$$ 

Let $f \colon \Xh \to \Yh$ be a morphism of simplicial presheaves. Then $f$ is called 
\begin{itemize}
\item a (local) weak equivalence if it induces stalkwise a weak equivalence of simplicial sets at every point of $\Manc$;
\item an injective cofibration if it induces a monomorphism of simplicial sets $\Xh(U) \to \Yh(U)$ for every object $U$ of $\Manc$;
\item a global fibration if it has the right lifting property with respect to any injective cofibration which is also a weak equivalence.
\end{itemize}

These classes of morphisms define a closed proper cellular simplicial model structure on $\sPre$ (see \cite[Theorem 2.3]{jardine}). We denote its homotopy category by $\hosPre$.

There are several other different interesting model structures on $\sPre$. In the local projective model structure the weak equivalences are again the local, i.e. stalkwise, weak equivalences. A morphism $f \colon \Xh \to \Yh$ of simplicial presheaves is called  
\begin{itemize}
\item a projective cofibration if it has the left lifting property with respect to maps which induce a trivial cofibration of simplicial sets $\Xh(U) \to \Yh(U)$ for every object $U$;
\item a local projective fibration if it has the right lifting property with respect to every projective cofibration which is also a local weak equivalence.
\end{itemize}

The classes of local weak equivalences, projective cofibrations and local projective fibrations provide $\sPre$ with the structure of a proper cellular simplicial model category (see \cite{blander} and \cite{dugger}). %The following fact shows that both model structures on $\sPre$ play a nice one-two with each other.

\begin{theorem}\label{comparemodel} {\rm (\cite{blander}, \cite{dugger}, \cite{dhi})} 
The identity functor is a left Quillen equivalence from the local projective model structure to the local injective model structure on $\sPre$.  
\end{theorem}

There is a third class of fibrations given by a local condition that is in fact easier to check and therefore convenient for applications. A map $f$ of simplicial presheaves is called a local fibration (resp. trivial local fibration) if the map of stalks $f_x$ is a Kan fibration (resp. Kan fibration and weak equivalence) of simplicial sets for every point $x$ of $\Manc$. Let $M$ be a complex manifold and let $U_{\bullet}$ be a simplicial presheaf on $\Manc$ with an augmentation map $U_{\bullet} \to M$ in $\sPre$. This map is called a {\it hypercover} of $M$ if it is a trivial local fibration and each $U_n$ is a coproduct of representables.

An interesting feature of the local projective structure is that one can detect its fibrant objects by a nice criterion (see \cite{dugger} and \cite{dhi}). Let $M$ be a complex manifold and let $U_{\bullet} \to M$ a hypercover of $M$. A simplicial presheaf $\Xh$ is said to satisfy descent for the hypercover $U_{\bullet} \to M$ if the natural map 
\[
\Xh(M) \to \holim_n\Xh(U_n)
\]
is a weak equivalence of simplicial sets. A simplicial presheaf $\Xh$ is fibrant in the local projective model structure if $\Xh(M)$ is a Kan complex for every object $M$ of $\Manc$ and if it satisfies descent for all hypercovers. We will make use of this criterion when we show that the singular functor is a fibrant replacement in $\sPre$.

Recall that a simplicial presheaf $\Xh$ is called locally fibrant if the unique map from $\Xh$ to the final object is a local fibration. 
For locally fibrant simplicial presheaves $\Xh$ and $\Yh$, let $\pi(\Xh,\Yh)$ be the quotient of $\Hom_{\sPre}(\Xh,\Yh)$ with respect to the equivalence relation generated by simplicial homotopies. The set $\pi(\Xh,\Yh)$ is called the set of simplicial homotopy classes of morphisms from $\Xh$ to $\Yh$. For a simplicial presheaf $\Xh$, denote by $\pi Triv/\Xh$ the category whose objects are the trivial local fibrations to $\Xh$ and whose morphisms are simplicial homotopy classes of morphisms which fit  in the obvious commutative triangle over $\Xh$. A crucial fact about the category $\pi Triv$ is that it approximates the homotopy category $\hosPre$ in the following sense.

\begin{prop}\label{verdierpure} {\rm (\cite[Proof of Theorem 2]{brown}, \cite[Proposition 2.1.13]{mv} and \cite[p. 55]{jardine})}
For any simplicial presheaves $\Xh$ and $\Yh$ with $\Yh$ locally fibrant, the canonical map
\[
\pi'(\Xh,\Yh):=\colim_{p \colon \Xh' \to \Xh \in \pi Triv/\Xh} \pi(\Xh',\Yh) \to \Hom_{\hosPre}(\Xh,\Yh)
\]
is a bijection.
\end{prop}
%
%%%%%%%%%%%%%%%%%%%%%%%%%%%%%%%%%%%%%%%%

We would like to be able to recover in $\hosPre$ homotopy classes of maps between topological spaces. But the full embedding that sends an object of $\Manc$ to the presheaf in simplicial dimension zero that it represents is very rigid. For we have an isomorphism
\[
\Hom_{\Manc}(X,Y) \cong \Hom_{\hosPre}(X,Y),
\]
i.e. the functor $\Manc \to \hosPre$ is still a full embedding. The simplicial direction is crucial and there are at least two ways to use it. The first one is the singular functor of Suslin and Voevodsky \cite{sv} in the topological context.

Let $\DD^n$ be the standard topological $n$-simplex 
\[
\DD^n =\{ (t_0,\ldots, t_n) \in \R^{n+1} | 0 \leq t_j \leq 1, \sum t_j =1 \}.
\]
For topological spaces $Y$ and $Z$, the singular function complex $\Sing(Z,Y)$ is the simplicial set whose $n$-simplices are continuous maps
\[
f \colon Z \times \DD^n \to Y.
\]
We denote the simplicial presheaf  
\[
Z\mapsto \Sing(Z,Y)=:\Sing Y(Z)
\]
by $\Sing Y$. Note that $\Sing Y$ is a simplicial presheaf on our site. The singular simplicial set of a topological space $Y$ is denoted by $\Sh(Y)\in \sS$. The geometric realization of a simplicial set $K$ is denoted by $|K|$.

For topological spaces $Y$ and $Z$, we denote by $Y^Z$ the topological space of continuous maps $Z\to Y$ with the compact-open topology. Then the adjunction of taking products and mapping spaces yields a canonical isomorphism of simplicial sets
\[
\Sing(Z,Y) = \Sh(Y^Z).
\]
\begin{lemma}\label{singfibrant}
Let $Y$ be a $CW$-complex. The simplicial presheaf $\Sing Y$ is a fibrant object in the local projective model structure on $\sPre$. 
\end{lemma}
\begin{proof}
Since the singular complex of any topological space is a Kan complex, $\Sing Y$ is objectwise fibrant. By \cite[Corollary 7.1]{dhi}, it remains to show that $\Sing Y$ satisfies descent for all hypercovers. So let $X$ be a complex manifold and $U_{\bullet} \to X$ a hypercover. By \cite[Theorem 1.3]{di}, the induced map $\hocolim_n U_n \to X$ is a weak equivalence. Since $Y$ is fibrant and $\Sh$ preserves weak equivalences, the induced map 
\[
\Sing Y(X) \to \Sing Y(\hocolim_n U_n)= \holim_n \Sing Y(U_n)
\]
is a weak equivalence too. Hence $\Sing Y$ satisfies descent for every hypercover and is fibrant.
\end{proof}

\begin{prop}\label{constant}
Let $K$ be a simplicial set. %We also denote by $K$ the constant simplicial presheaf with value $K$.  
The natural map  
\[
K \to \Sing |K|
\]
is a weak equivalence of simplicial presheaves.  
\end{prop}
\begin{proof}
The stalks of the simplicial presheaf $K$ are all canonically isomorphic to $K$ itself. In order to compute the stalks of $\Sing |K|$, $x$ be a point in a complex manifold $X$ and let $D_{\epsilon}$ be a small open disk of radius $\epsilon$ containing $x$. Then we have a canonical map
\[
\colim_{\epsilon \to 0}\Hom(\DD^{\bullet}, |K|^{D_{\epsilon}}) \to \Hom(\DD^{\bullet},|K|).
\]
Since $X$ is a manifold and hence locally contractible, the left hand side is isomorphic to the stalk $(\Sing |K|)_x$ at $x$. The right hand side is just the singular complex of the topological space $|K|$. Hence we have a map
\[
(\Sing |K|)_x \to \Sh(|K|).
\]
This map is a weak equivalence. This implies that the map $K \to (\Sing |K|)_x$ is a weak equivalence. Hence $K\to \Sing |K|$ is a weak equivalence of simplicial presheaves. 
%
%Let $V_{\bullet} \in \sCW$ be such a representable hypercovering of $X$. Since $V_{\bullet}$ is a representable simplicial presheaf, the set of maps $V_{\bullet} \to K$ in $\sPre$ is in bijection with the set of levelwise locally constant maps $V_{\bullet} \to K$. Since each $V_n$ is locally contractible, the set of homotopy classes of levelwise locally constant maps $V_{\bullet} \to K$ is in bijection with the set of simplicial homotopy classes of maps $\Sh(V_{\bullet}) \to K$ of simplicial sets, where $\Sh(V_{\bullet})$ denotes the diagonal simplicial set obtained by applying the functor $\S$ levelwise.  
\end{proof}

Lemma \ref{singfibrant} and Proposition \ref{constant} show that $\Sing |K|$ is a fibrant replacement of $K$ in the local projective model structure on $\sPre$.

\begin{lemma}\label{Singwe}
The functor $\Sing \colon \sS \to \sPre$, $K \mapsto \Sing |K|$, sends weak equivalences in $\sS$ to local weak equivalences. 
\end{lemma}
\begin{proof}
This follows immediately from the calculation of stalks in the proof of Proposition \ref{constant}. 
\end{proof}

\begin{prop}\label{singpretop}
Let $X$ be a complex manifold and $K$ be a simplicial set. There is a natural bijection 
\[
\Hom_{\hosPre}(X, \Sing |K|) \cong \Hom_{\hoTop}(X,|K|).
\]
\end{prop}
\begin{proof}
There are different ways to see this. One would be to apply Verdier's hypercovering theorem, which we recalled in Proposition \ref{verdierpure}. This is possible since $\Sing |K|$ is stalkwise fibrant. Another way is to use the freedom to pick one of the various underlying model structures for the homotopy category of $\sPre$. In the local projective model structure, $X$ is a cofibrant object, since it is representable. By Lemma \ref{singfibrant}, $\Sing |K|$ is fibrant in the local projective structure. Then we get the following sequence of natural bijections 
\[
\Hom_{\hosPre}(X, \Sing |K|) \cong \pi_0(\Sing |K|(X)) \cong \pi_0(\Sh(|K|^X) \cong \Hom_{\hoTop}(X,|K|).
\]
%
%Verdier's hypercovering theorem implies that there is a natural bijection 
%$$\Hom_{\hosPre}(X, \Sing Y) \cong \colim_{\Uh \in \Hyp_{X}} \pi(\Uh, \Sing Y)$$
%where the colimit is taken over all hypercoverings of the representable simplicial presheaf $X$. Since $X$ is a complex manifold, it suffices to take the colimit on the right hand side over representable hypercoverings. 
%
%Let $V_{\bullet} \in \sCW$ be such a representable hypercovering of $X$. Since $V_{\bullet}$ and $Y^{\DD^{\bullet}}$ are representable simplicial presheaves, the set of maps $V_{\bullet} \to Y^{\DD^{\bullet}}$ in $\sPre$ is in canonical bijection with the set of maps $V_{\bullet} \to Y^{\DD^{\bullet}}$ in the category of simplicial topological spaces. 
%
%By \cite[Theorem 4.3]{di}, the canonical map $|V_{\bullet}| \to X$ is a weak equivalence of topological spaces. Using the isomorphisms (\ref{singadjunction}), (\ref{geomadjunction}) and (\ref{representedhtpy}), the above colimit is isomorphic to a constant diagram with values 
%$$\pi_{\CW}(X,Y) \cong \Hom_{\hoTop}(X,Y).$$
\end{proof}

As a consequence of Propositions \ref{constant} and \ref{singpretop}, we obtain the following result about maps in the homotopy category to a simplicial presheaf given by a simplicial set.

\begin{cor}\label{corconstant}
Let $X$ be a complex manifold and $K$ be a simplicial set. There is a natural bijection 
\[
\Hom_{\hosPre}(X, K) \cong \Hom_{\hoTop}(X,|K|).
\]
\end{cor}

For given $n\geq 0$ and an abelian group $A$, let $K(A,n)$ be an Eilenberg-MacLane space in $\sS$. For a topological space $X$, we denote by $H_{\sing}^n(X; A)$ the $n$th singular cohomology with coefficients in $A$. It is isomorphic to the group of maps $X \to |K(A,n)|$ in the homotopy category $\hoTop$. From \cite[Corollary D.13]{hs}, we can read off the homotopy groups of the global sections of $\Sing |K(A,n)|$. 

\begin{lemma}\label{D.13}
For any complex manifold $X$, the $i$th homotopy group of the simplicial set $\Sing |K(A,n)|(X)$ is given by a natural isomorphism 
\[
\pi_i (\Sing |K(A,n)|(X)) \cong H^{n-i}_{\sing}(X;A).
\]
\end{lemma}

Moreover, Corollary \ref{corconstant} shows that $K(A,n)$ represents singular cohomology for complex manifolds also in $\hosPre$.
\begin{prop}%\label{D.13}
For an abelian group $A$ and a complex manifold $X$, there is a natural isomorphism 
\[
\Hom_{\hosPre}(X, |K(A,n)|) \cong H_{\sing}^n(X; A).
\]
\end{prop}
\begin{proof}
By Propositions \ref{constant} and \ref{singpretop}, the set $\Hom_{\hosPre}(X, K(A,n))$ is in bijection with $\Hom_{\hoTop}(X, |K(A,n)|)\cong H_{\sing}^n(X; A)$. %Moreover, since $X$ is locally contractible, the sheaf cohomology group $H^n(X; A)$ is isomorphic to the singular cohomology group $H_{\sing}^n(X;A)$. 
\end{proof}

\begin{remark}\label{constantdoldkan}
For an abelian group $A$, we also denote by $A$ the constant presheaf with value $A$. We denote the image of the complex with a single nontrivial sheaf $A$ placed in degree $n \geq 0$ under the Dold-Kan correspondence by $K(A,n)$. This is just the constant simplicial presheaf with value the simplicial Eilenberg-MacLane space $K(A,n)$. 
\end{remark}

%%%%%%%%%%%%%%%%%%%%%%%%%%%%%%%%%%%

\subsection{Hypercohomology}

For a  chain complex of presheaves of abelian groups $\Ch_*$ on $\Manc$, we denote by $H_i(\Ch^*)$ the presheaf $U\mapsto H_i(\Ch_*(U))$. For a cochain complex $\Ch^*$ we will denote by $\Ch_*$ its associated chain complex given by $\Ch_n:=\Ch^{-n}$. For any given $n$, we denote by $\Ch^*[n]$ the cochain complex given in degree $q$ by $\Ch^q[n]:=\Ch^{q+n}$. The differential on $\Ch^*[n]$ is the one $\Ch^*$ multiplied by $(-1)^n$.

Applying the normalized chain complex functor pointwise we obtain a functor $\Gh \mapsto N(\Gh)$ from simplicial presheaves of abelian groups to chain complexes of presheaves of abelian groups. Then we have $\pi_i(\Gh)\cong H_i(N(\Gh))$. Moreover, the functor has a right adjoint $\Gamma$ again obtained by applying the corresponding functor for chain complexes pointwise. The following result is the analog of the Dold-Kan correspondence for simplicial sheaves.

\begin{prop}\label{1.24} {\rm (\cite{brown} and \cite[Proposition 2.1.24]{mv})} 
The pair $(N,\Gamma)$ is a pair of mutually inverse equivalences between the category of complexes of presheaves of abelian groups $\Ch_*$ with $\Ch_i=0$ for $i<0$ and the category of simplicial presheaves of abelian groups.
\end{prop}

Let $\Fh$ be a presheaf of abelian groups. We denote the image of the complex with a single nontrivial presheaf $\Fh$ placed in degree $n$ under the Dold-Kan correspondence by $K(\Fh,n)$. The $n$th cohomology $H^n(X; \Fh)$ of $X$ with coefficients in $\Fh$ is defined as the sheaf cohomology of $X$ with coefficients in the sheaf $a\Fh$ associated to $\Fh$.

If $\Ch^*$ is a cochain complex of presheaves of abelian groups on $\mathbf{T}$, the hypercohomology $H^*(U, \Ch^*)$ of an object $U$ of $\Manc$ with coefficients in $\Ch^*$ is the graded group of morphisms $\Hom(\Z_U,a\Ch^*)$ in the derived category of cochain complexes of sheaves on $\Manc$, where $a\Ch^*$ denotes the complex of associated sheaves of $\Ch^*$.

The following result is a version of Verdier's hypercovering theorem due to Ken Brown in \cite[Theorem 2]{brown}. We refer the reader also to \cite[Proposition 2.1.25]{mv} and \cite{jardineverdier} for more details.

\begin{prop}\label{verdier} {\rm (\cite[Theorem 2]{brown})}
Let $\Ch^*$ be a cochain complex of presheaves of abelian groups on $\Manc$. Then for any integer $n$ and any object $U$ of $\Manc$ one has a canonical isomorphism 
\[
H^n(U; \Ch^*)\cong \Hom_{\hosPre}(U, \Gamma(\Ch^*[-n])).
\]
In particular, if $\Ch_*=\Fh$ is a presheaf of abelian groups, we have 
\[
H^n(U;\Fh)\cong \Hom_{\hosPre}(U, K(\Fh,n)).
\]
\end{prop}

%%%%%%%%%%%%%%%%%%%%%%%%%%%%%%%%%%

\subsection{Deligne cohomology in terms of simplicial presheaves}

For $p\geq 0$, let $K(\Z,n) \to K(\C,n)$ be a map of simplicial sets whose cohomology class corresponds to the $(2\pi i)^p$-multiple of the canonical inclusion $\Z \subset \C$, under the isomorphism
\[
\Hom_{\hosS}(K(\Z,n),K(\C,n))\cong \Hom(\Z, \C).
\]
This induces a morphism of simplicial presheaves
\[
K(\Z,n) \to K(\C,n).
\]
As mentioned in Remark \ref{constantdoldkan},  the constant simplicial presheaf $K(\C,n)$ is the image of the Dold-Kan correspondence of the complex $\C$ given by the constant presheaf $\C$ in degree $n$. The canonical inclusion into the complex of sheaves $\Omega^*[-n]$ of holomorphic forms induces a map of simplicial presheaves
\[
K(\C,n) \to \Gamma(\Omega^*[-n]).
\]
Combining these maps we obtain a morphism of simplicial presheaves
\[
K(\Z,n) \to \Gamma(\Omega^*[-n]).
\]
We define $K(\Z,n)(p)$ to be the homotopy pullback of the diagram of simplicial presheaves  
\begin{equation}\label{exampleH}
\xymatrix{
K(\Z,n)(p) \ar[d] \ar[r] & K(\Z, n) \ar[d] \\
\Gamma(\Omega^{*\geq p}[-n]) \ar[r] & \Gamma(\Omega^*[-n]).}
\end{equation}
For a simplicial presheaf $\Xh$, we call the abelian groups 
\[
\Hom_{\hosPre}(\Xh, K(\Z,n)(p))
\]
the Hodge filtered $K(\Z,n)$-cohomology groups of $\Xh$.

\begin{prop}\label{D.6}
Let $X$ be a complex manifold. The Hodge filtered $K(\Z,n)$-cohomology groups of $X$ agree with Deligne cohomology groups of $X$, i.e. for every $n\geq 0$ and $p\geq 0$, the map induced by the adjointness property of the Dold-Kan correspondence induces an isomorphism   
\[
\Hom_{\hosPre}(X, K(\Z,n)(p)) \cong H_{\Dh}^{n}(X;\Z(p)).
\]
\end{prop}
%\begin{proof}
%By its definition as a homotopy pullback, $K(\Z,n)(p)$ sits in a long exact sequence
%$$\begin{array}{rl}
%\ldots \to H^{n-1}(X; \C) & \to \Hom_{\hosPre}(X, K(\Z,n)(p)) \to \\
%\to H^{n}(X;\Z) \oplus H^n(X; \Omega^{*\geq p}) & \to H^n(X; \C) \to \ldots
%\end{array}$$ 
%and the equivalence $K(\Z,n) \simeq \Omega K(\Z,n+1)$, implies that we have an isomorphism 
%$$\Hom_{\hosPre}(X, K(\Z,n)(p)) \cong \Hom_{\hosPre}(X, \Omega K(\Z,n+1)(p)) \cong ??.$$
%
%The Dold-Kan correspondence yields a map to the corresponding long exact sequence for Deligne cohomology. Since all maps but the one in question are isomorphisms, this proves the claim. 
%\end{proof}

This result follows from the fact that the Dold-Kan correspondence induces a map of long exact sequences in which we know that all the maps but the one in question are isomorphisms. We will discuss this in a more general context in the next section (see Proposition \ref{les} below).

%%%%%%%%%%%%%%%%%%%%%%%%%%%%%%%%%%%%%%%%

\section{Stable homotopy theory for complex manifolds}

\subsection{Presheaves of symmetric spectra}

Although most of the theory can be formulated in terms of simplicial presheaves, we need a more general setting in order to obtain a product structure in the next section. Therefore, we consider presheaves of symmetric spectra. \\

Let $\mathbf{T}$ be a Grothendieck site whose underlying category is small. Let $\sPrep$ be the category of pointed simplicial presheaves on $\mathbf{T}$. A symmetric sequence in $\sPrep$ is a sequence of pointed simplicial presheaves $\Xh_0, \Xh_1, \ldots$ with an action by the $n$th symmetric group $\Sigma_n$ on $\Xh_n$. For two symmetric sequences $\Xh$ and $\Yh$, the tensor product $\Xh \otimes \Yh$ is defined as the symmetric sequence given in degree $n$ by 
\[
(\Xh \otimes \Yh)_n = \coprod_{p+q=n} \Sigma_n \times_{\Sigma_p \times \Sigma_q} (\Xh_p \wedge \Yh_q).
\]

Let $S^n$ denote the $n$-fold smash product of the simplicial circle 
\[
S^1=\Delta^1/\partial\Delta[1]
\]
with itself. Considering $S^n$ as a constant simplicial presheaf, we obtain a symmetric sequence $(S^0, S^1, S^2, \ldots)$ where $\Sigma_n$ acts on $S^n$ by permutation of the factors. We denote this symmetric space by $\Sph$. 

A presheaf of symmetric spectra $\Eh$ is given by a sequence of pointed presheaves $\Eh_n$ with an action $\Sigma_n$ for $n\geq 0$ together with $\Sigma_n$-equivariant maps $S^1 \wedge \Eh_n \to \Eh_{n+1}$ such that the composite 
\[
S^p \wedge \Eh_n \to S^{p-1} \wedge \Eh_{n+1} \to \ldots \to \Eh_{n+p}
\]
is $\Sigma_p \times \Sigma_n$-equivariant for all $n,p \geq 0$. A map $\Eh \to \Fh$ of presheaves of symmetric spectra is a collection of $\Sigma_n$-equivariant maps $\Eh_n \to \Fh_n$ compatible with the structure maps of $\Eh$ and $\Fh$. Hence we could say that a presheaf of symmetric spectra is a symmetric spectrum object in the category $\sPrep$. This justifies denoting the category of presheaves of symmetric spectra on $\mathbf{T}$ by 
\[
\SpsPre(\mathbf{T}) = \SpsPre.
\]

Another way to think of a presheaf of symmetric spectra $\Eh$ is that $\Eh$ is a symmetric space together with the structure of an $\Sph$-module given by a map 
\[
m \colon \Sph \otimes \Eh \to \Eh.
\]

The important advantage of symmetric spectra compared to, say, presheaves of Bousfield-Friedlander spectra is that $\SpsPre$ has a monoidal structure defined as follows. For two presheaves of symmetric spectra $\Eh$ and $\Fh$, the smash product $\Eh \wedge \Fh$ is defined as the coequalizer 
\[
\Eh \otimes \Sph \otimes \Fh \rightrightarrows \Eh \otimes \Fh \to \Eh \otimes_{\Sph}\Fh=:\Eh \wedge \Fh.
\]
The two maps in the diagram are given by the module structure of $\Fh$ and the twisted module structure map of $\Eh$
\[
\Eh \otimes \Sph \xrightarrow{\tau} \Sph \otimes \Eh \to \Eh.
\]

\begin{example}
As for symmetric spectra in \cite{hss}, we have the following basic examples of presheaves of symmetric spectra.

(a) Let $\sSp$ be the category of symmetric spectra of simplicial sets defined in \cite{hss}. Every symmetric spectrum $E\in \sSp$ defines a presheaf of symmetric spectra given by the constant presheaf with value $E$. In particular, the sequence of constant simplicial presheaves $(S^0, S^1, S^2, \ldots)$ with the obvious structure maps also defines a presheaf of symmetric spectra which we also denote by $\Sph$ and call it the symmetric sphere spectrum. 

(b) For a pointed simplicial presheaf $\Xh$, we denote by $\Sigma^{\infty}\Xh$ the presheaf of symmetric spectra given by the sequence of pointed simplicial presheaves $S^n\wedge \Xh$ with the natural isomorphisms $S^1 \wedge S^n \wedge \Xh \to S^{n+1}\wedge \Xh$ and the diagonal action of $\Sigma_n$ on $S^n\wedge \Xh$ coming from the left permutation action on $S^n$ and the trivial action on $\Xh$. 

(c) In particular, if $X$ an object of $\mathbf{T}$, we can associate to $X$ a presheaf of symmetric spectra $\Sigma_+^{\infty}(X)$, where $X$ denotes the simplicial presheaf represented by $X$ and the $+$-subscript means that we add a disjoint basepoint.
 
(d) Furthermore, if $\Eh$ is a presheaf of symmetric spectra and $n$ an integer, then we denote by $\Sigma^n\Eh$ the $n$th suspension of $\Eh$ whose $k$th space is $\Eh_{k+n}$.

(e) Let $\Xh$ and $\Yh$ be presheaves of symmetric spectra. There is a function spectrum $\Homhs(\Xh,\Yh)$ in $\SpsPre$ defined as the limit of the diagram in $\SpsPre$
$$\Hom_{\Sigma}(\Xh,\Yh) \rightrightarrows \Hom_{\Sigma}(\Sph \otimes \Xh, \Yh)$$
where the two arrows given by $m^*$ and $m_*$ respectively for the $\Sph$-module structure map $\Sph \otimes \Xh \to \Xh$. The endofunctor $\Homhs(\Yh,-)$ of $\SpsPre$ is right adjoint to the functor $- \wedge \Yh$. 
\end{example}

The category $\SpsPre$ has a stable model structure defined in two steps. We equip $\sPrep$ either with the local injective or with the local projective model structure. In either case we obtain the following intermediate structure.

\begin{defn}\label{defstrict}
A map $f \colon \Eh \to \Fh$ in $\SpsPre$ is called a projective weak equivalence (respectively fibration) if each $f_n \colon \Eh_n \to \Fh_n$ is a weak equivalence (respectively fibration) in $\sPrep$. A map is called a projective cofibration if it has the left lifting property with respect to all maps that are projective weak equivalences and projective fibrations.
\end{defn}

For either the local injective or the local projective model structure on $\sPrep$, the following result is a consequence of Hovey's general result \cite[Theorems 8.2 and 8.3]{hovey}.

\begin{prop}\label{propstrict}
The classes of projective weak equivalence, projective fibrations and projective cofibrations define a proper cellular model structure on $\SpsPre$ such that $\SpsPre$ is a $\sSp$-model category.
\end{prop}

%%%%%%%%%%%%%%%%%%%%%%%%%%%%%%%%%

\begin{defn}\label{stablemodel} {\rm (\cite[Definition 8.7]{hovey})} 
Define the set $S$ of maps in $\SpsPre$ to be 
\[
\{ F_{n+1}(C\wedge S^1) \xrightarrow{\zeta_n^C} F_nC \}
\]
as $C$ runs through the domains and codomains of the generating cofibrations of $\sPrep$, and $n\geq 0$, where the map $\zeta_n^C$ is adjoint to the map
\[
C\wedge S^1 \to \Ev_{n+1}F_nC = \Sigma_{n+1}\times (C\wedge S^1)
\]
corresponding to the identity of $\Sigma_{n+1}$. 
The {\it stable local injective model structure} (respectively {\it stable local projective model structure}) on $\SpsPre$ is defined to be the left Bousfield localization with respect to $S$ of the strict local injective (respectively local projective) model structure. 
\end{defn}

Since both the local projective and local injective model structures are proper, simplicial and cellular whose generating cofibrations and trivial cofibrations have cofibrant domains and codomains, we can deduce the following consequence of \cite[Theorem 9.3]{hovey}, and Theorem \ref{comparemodel}.

\begin{theorem}\label{stablecomparemodel} {\rm (\cite{blander}, \cite{dugger}, \cite{hovey}, \cite{jardinestable})} 
The identity functor on $\SpsPre$ is a left Quillen equivalence from the stable local projective model structure to the stable local injective model structure. We denote the stable homotopy category obtained by localizing $\SpsPre$ at the stable local equivalences by $\hoSpsPre$. 
\end{theorem}

\begin{defn}\label{omegaspectra}
{\rm (1)} A presheaf of symmetric spectra $\Eh \in \SpsPre$ is called an $\Omega$-spectrum if each $\Eh_n$ is fibrant in $\sPrep$ and the adjoint structure maps $\Eh_n \to \Eh_{n+1}^{S^1}$ are weak equivalences in $\sPrep$ for all $n \geq 0$. \\
{\rm (2)} A map in $\SpsPre$ is called an injective fibration if it has the right lifting property with respect to all maps that are level cofibrations and level weak equivalences.\\
{\rm (3)} A spectrum $\Eh$ in $\SpsPre$ is called an injective spectrum if the map $\Eh \to *$ is an injective fibration. 
\end{defn}

By \cite[Theorem 8.8]{hovey}, the stably fibrant objects are the $\Omega$-spectra. For two presheaves of symmetric spectra $\Xh$ and $\Yh$, let $\Map_{\SpsPre}(\Xh, \Yh)$ denote the mapping space which is part of the simplicial structure on $\SpsPre$.
The following lemma shows that the identity functor on $\SpsPre$ is a Quillen equivalence from the stable local injective model structure to the stable model structure of \cite{jardinestable}.

\begin{lemma}\label{3.1.4}
Let $f \colon \Xh \to \Yh$ be a map in $\SpsPre$. Then the following conditions are equivalent:\\
{\rm (1)} $f$ is a stable equivalence. \\ 
{\rm (2)} $f$ induces a weak equivalence of Kan complexes 
\[
\Map_{\SpsPre}(f,\Eh): \Map_{\SpsPre}(\Yh,\Eh) \to \Map_{\SpsPre}(\Xh,\Eh)
\]
for every injective $\Omega$-spectrum $\Eh$ in $\SpsPre$.\\
{\rm (3)} $f$ induces a level equivalence 
\[
\Homhs(f,\Eh) \colon \Homhs(\Yh,\Eh) \to \Homhs(\Xh,\Eh)
\]
for every injective $\Omega$-spectrum $\Eh$ in $\SpsPre$.
\end{lemma}
\begin{proof}
In order to prove the lemma we recall the injective model structure on $\SpsPre$ of \cite[Theorem 2]{jardinestable}. A map in $\SpsPre$ is called an injective cofibration (injective weak equivalence) if it is a levelwise cofibration (weak equivalence). The fibrant objects in the injective model structure are the injective $\Omega$-spectra. 

The identity functor provides a Quillen equivalence between the injective and projective model structures on $\SpsPre$. In particular, every injective $\Omega$-spectrum is a fibrant object in the projective model structure on $\SpsPre$ and every projective cofibrant object is also injective cofibrant. This has the following consequence. 

%Let $Q_{\inj}$ (respectively $Q_{\proj}$) be a cofibrant replacement functor in the injective (respectively projective) model structure. 
Let $Q_{\proj}$ be a cofibrant replacement functor in the projective model structure. If $f$ is a stable equivalence, then $\Map_{\SpsPre}(Q_{\proj}f, \Eh)$ is a weak equivalence of Kan complexes for every injective $\Omega$-spectrum $\Eh$. Since the maps $\Xh \to Q_{\proj}\Xh$ and $\Yh \to Q_{\proj}\Yh$ are level equivalences, we obtain that $\Map_{\SpsPre}(f, \Eh)$ is a weak equivalence of Kan complexes for every injective $\Omega$-spectrum $\Eh$. 

Now let $R_{\inj}$ be a fibrant replacement functor in the injective model structure and assume that $f$ induces a weak equivalence $\Map_{\SpsPre}(f, \Eh)$ for every injective $\Omega$-spectrum $\Eh$. This implies that $\Map_{\SpsPre}(Q_{\proj}f, R_{\inj}\Fh)$ is a weak equivalence for every $\Omega$-spectrum $\Fh$. Since the map $\Fh \to R_{\inj}\Fh$ is a level equivalence, this shows that $\Map_{\SpsPre}(Q_{\proj}f, \Fh)$ is a weak equivalence for every $\Omega$-spectrum $\Fh$, i.e. that $f$ is a stable equivalence. This proves that (1) and (2) are equivalent.

The second and third conditions are equivalent, since 
$$\Ev_k \Homhs(f, \Eh) = \Map_{\SpsPre}(f, \Homhs(F_k S^0, \Eh)).$$
\end{proof}

Consider $\SpsPre$ with the stable local injective model structure. Let $K$ be the class in $\SpsPre$ of all maps $f \wedge \Xh$, where $f$ is a stable trivial cofibration and $\Xh$ is a presheaf of symmetric spectra. One can show exactly as in \cite[\S 5.4]{hss}, that that stable local injective model structure is monoidal and that each map in $K$-cof is a stable equivalence. This fact implies that the monoid axiom holds in $\SpsPre$ and by \cite[Theorem 4.1]{ss}, we can deduce the following result on monoids in $\SpsPre$. 
\begin{theorem}
The forgetful functor creates a model structure on the category of monoids in $\SpsPre$ for which a morphism is a weak equivalence (fibration) if and only if the underlying map of presheaves of symmetric spectra is a stable equivalence (stable fibration). 
\end{theorem}
\begin{remark}\label{constantembedding}
The functor $\sSp \to \SpsPre$ sending a symmetric spectrum $E$ to the constant presheaf of symmetric spectra with value $E$ is a strong symmetric monoidal left Quillen functor. In particular, a commutative $\Sph$-algebra $E$ in $\sSp$ is still a commutative $\Sph$-algebra in $\SpsPre$ when we consider $E$ as a constant presheaf.  
\end{remark}
%
%%%%%%%%%%%%%%%%%%%%%%%%%%%%%%%%%%%%%%%%%%%%%
%
%
\subsection{Eilenberg-MacLane spectra}
In the remainder of this section, we start to explore Brown's idea \cite{brown} to use sheaves of spectra to define generalized sheaf cohomology.

Let $\Predga$ be the category of presheaves of differential graded $\C$-algebras on $\mathbf{T}$. 
By \cite{shipley}, for every differential graded algebra $A_*$, there is functorial construction of a symmetric Eilenberg-MacLane ring spectrum $HA_*$. (The reader may also want to consult \cite[\S 2.6]{ds} for a short summary on the construction of the functor $H$.) Pointwise application yields a functor 
\[
H \colon \Predga \to \SpsPre.
\]

By its construction, the functor $H$ sends the constant presheaf $\C$ to the constant symmetric Eilenberg-MacLane spectrum $H\C$. Moreover, since the tensor product in $\Predga$ and the smash product in $\SpsPre$ are defined pointwise, the functor $H$ is a lax monoidal functor and its image lies in fact in the subcategory of monoids over $H\C$, or in other words $H\C$-algebras in $\SpsPre$: 
\[
H \colon \Predga \to \HCAlg.
\]
In particular, a morphism $\Ch^* \to \Dh^*$ of presheaves of differential graded $\C$-algebras induces a morphism of monoids $H\Ch^* \to H\Dh^*$ in $\SpsPre$.

Our applications in the following sections require graded versions of the complex of holomorphic forms, since different coefficient rings will be substituted in the game. The correct choice of compatible  gradings and filtrations is an important point. This will become apparent in the proof of the main result in Theorem \ref{thmfundses}.

So we enlarge our roster by considering also cochains, cocycles and forms with values in an evenly graded complex vector space $\Vh$. The main example will be $\Vh_{2*}=\pi_{2*}MU\otimes \C$. We will use the convention to grade cochains and forms with values in $\Vh$ in such a way that $C^i(X; \Vh_{j})$ has total degree $(i-j)$. We will write
\[
C^n(X; \Vh_{2*}) := \bigoplus_{j} C^{n+2j}(X; \Vh_{2j}),
\]
and
\[
H^n(X; \Vh_{2*}) := \bigoplus_{j} H^{n+2j}(X; \Vh_{2j}).
\]
For the presheaf of holomorphic forms we will write 
\[
\Omega^*(\Vh_{2*}) = \bigoplus_j \Omega^{*}(\Vh_{2j})[-2j].
\] 
For a given integer $p$, we will denote  
\[
\Omega^{*\geq p}(\Vh_{2*}) := \bigoplus_j \Omega^{*\geq p+j}(\Vh_{2j})[-2j].
\]
%the subcomplex of holomorphic forms on $X$ of the graded complex $\Omega^{*}(X; \Vh_{2*})$ such that, for every $j$, the $j$th component is given by 
%$$\Omega^{*\geq p+j}(X; \Vh_{2j})[-2j].$$
%
Note that by our grading convention, this defines in general a complex different from $\Omega^*(\Vh_{2*})$ even for negative $p$. 

For hypercohomology groups we write 
\[
H^n(X; \Omega^{*\geq p}(\Vh_{2*})) = \bigoplus_{j} H^{n+2j}(X; \Omega^{*\geq p+j}(\Vh_{2j})[-2j]).
\]
These gradings have corresponding counterparts on the level of Eilenberg-MacLane spectra. Let $\Vh_{2*}$ be an evenly graded $\C$-algebra and $\Ch^* \in \Predga$. We denote by $\Ch^*(\Vh_{2*})$ the presheaf of complexes given in degree $n$ by
\[
\Ch^n(\Vh_{2*}) = \bigoplus_j \Ch^{n+2j} \otimes \Vh_{2j}.
\]
The image of $\Ch^*(\Vh_{2*})$ under $H$ in $\SpsPre$ has the form  
\[
H(\Ch^*(\Vh_{2*})) = \bigvee_j \Sigma^{2j} H(\Ch^{*} \otimes \Vh_{2j}).
\]
In particular, if $\Ch^*(\Vh_{2*})$ is just the graded ring $\Vh_{2*}$ considered as a chain complex, then we have 
\[
H(\Vh_{2*}) = \bigvee_j \Sigma^{2j}H(\Vh_{2j}).
\]
For example, let $\Ch^*=\Omega^*$ be the presheaf of holomorphic forms on $\Manc$. The inclusion of cochain complexes $\C \into \Omega^*$ induces a map of presheaves of symmetric spectra
\[
H\C \to H\Omega^*.
\]
Moreover, for the $\C$-algebra $\Vh_{2*}$, there is a canonical map of presheaves of symmetric spectra
\[
H(\Vh_{2*}) \to H(\Omega^*(\Vh_{2*})).
\]

%%%%%%%%%%%%%%%%%%%%%%%%%%%%%%%%%%%%

\begin{prop}\label{hypersymmgraded}
Let $X$ be an object of $\mathbf{T}$, $\Ch^* \in \Predga$ and let $\Vh_{2*}$ be an evenly graded commutative $\C$-algebra. Then we have a natural isomorphism of graded-commutative rings
\[
\bigoplus _n \Hom_{\hoSpsPre}(\Sigma^{\infty}_+(X), \Sigma^nH(\Ch^*(\Vh_{2*}))) \cong \bigoplus_n H^n(X; \Ch^*(\Vh_{2*})).
\]
In particular, for $X\in \Manc$, there are natural isomorphisms of graded commutative rings
\[
\bigoplus_n\Hom_{\hoSpsPre}(\Sigma^{\infty}_+(X), \Sigma^n H(\Omega^*(\Vh_{2*}))) \cong \bigoplus_n H^n(X; \Vh_{2*})~\mathrm{and}
\]
\[
\bigoplus_n\Hom_{\hoSpsPre}(\Sigma^{\infty}_+(X), \Sigma^n H(\Omega^{*\geq p}(\Vh_{2*}))) \cong \bigoplus_n H^n(X; \Omega^{*\geq p}(\Vh_{2*})).
\]
\end{prop}
\begin{proof}
Let $H_{\nomo} \colon \mathrm{Ch}_{\Z} \to \sSp$ be the non-monoidal Eilenberg-MacLane spectrum functor that sends a chain complex $C$ to a symmetric spectrum which represents cohomology with coefficients in $C$. By \cite[Proposition 5.1]{ss2}, the canonical map $H_{\nomo}C \to HC$ is an equivalence of underlying symmetric spectra. Hence, for every object $U \in \mathbf{T}$, the canonical map
\[
H_{\nomo}(\Ch^*(\Vh_{2*}))(U) \to H(\Ch^*(\Vh_{2*}))(U)
\]
is an equivalence of symmetric spectra. By abuse of notation, we denote the pointwise application of $H_{\nomo}$ again by $H_{\nomo}$. Then the induced map of presheaves of symmetric spectra 
\begin{equation}\label{Hequiv}
H_{\nomo} (\Ch^*(\Vh_{2*})) \to H(\Ch^*(\Vh_{2*}))
\end{equation}
is a stable equivalence in $\SpsPre$. The spectrum $H_{\nomo}(\Ch^*(\Vh_{2*}))$ is a ring in the stable homotopy category of presheaves of spectra. Moreover, equivalence (\ref{Hequiv}) induces an isomorphism of rings in the stable homotopy category of presheaves of spectra between $H_{\nomo}(\Ch^*(\Vh_{2*}))$ and $H(\Ch^*(\Vh_{2*}))$.   
Hence it suffices to show that 
\[
\Hom_{\hoSpsPre}(\Sigma^{\infty}_+(X), \Sigma^nH_{\nomo}(\Ch^*(\Vh_{2*}))) \cong H^n(X; \Ch^*(\Vh_{2*}))
\]
is an isomorphism. 

Since $\Sigma^{\infty}$ is left Quillen adjoint to the evaluation functor at the $0$th space, it suffices to show that we have a natural isomorphism
\[
\Hom_{\hosPre}(X, \Gamma(\Ch^*(\Vh_{2j})[-n])) \cong H^{n-2j}(X; \Ch^*(\Vh_{2j})).
\]
But, using our grading conventions, this is the content of Proposition \ref{verdier}. 
\end{proof}

%
%%%%%%%%%%%%%%%%%%%%%%%%%%%%%%%%%%%%%

\subsection{The singular functor for symmetric spectra}

In this subsection we work again on the site $\mathbf{T}=\Manc$. Our goal now is to show that a constant presheaf of symmetric spectra $E$ represents $E$-cohomology also in $\hoSpsPre$. There are different ways to show this. We use the singular functor as our favorite tool.

Let $\sSpt$ be the category of symmetric spectra of topological spaces. Let $X$ be a complex manifold and $F\in \sSpt$. The levelwise defined spectrum $F^X$ of functions from $X$ to $F$ is again a symmetric spectrum of topological spaces. Applying the singular simplicial set functor levelwise yields then a symmetric spectrum of simplicial sets $\Sh(F^X)$ given in degree $n$ by $\Sh(F_n^X)$. We denote the object in $\SpsPre$, defined by the functor 
\[
X \mapsto \Sh(F^X)
\]
by $\Sing F$. Moreover, for $E\in \sSp$ let $|E|$ be the geometric realization of $E$ in $\sSpt$. Then $\Sing$ also defines a functor
\[
\Sing \colon \sSp \to \SpsPre, ~ E\mapsto \Sing |E|.
\]
 
\begin{prop}\label{stableconstant}
Let $E\in \sSp$ be a symmetric spectrum. We also denote by $E$ the constant presheaf with value $E$. The natural map  
\[
E \to \Sing |E|
\]
is a levelwise equivalence and hence a stable equivalence in $\SpsPre$.  
\end{prop}
\begin{proof}
By Proposition \ref{constant}, each map $E_n \to \Sing |E_n| \cong (\Sing |E|)_n$ is a weak equivalence in $\sPrep$. Hence $E \to \Sing |E|$ is a levelwise equivalence in $\SpsPre$.
\end{proof}
\begin{lemma}\label{stablesingfibrant}
For any spectrum $E\in \sSp$, the presheaf of symmetric spectra $\Sing |E|$ is fibrant in the stable local projective model structure on $\SpsPre$.
\end{lemma}
\begin{proof}
By Lemma \ref{singfibrant}, each space $\Sing |E|_n$ is fibrant in the local projective model structure on $\sPrep$. Moreover, for any complex manifold $X$, the induced map 
$$\Sh(|E_n|^X) \to \Sh((\Omega |E_{n+1}|)^X) \cong \Omega(\Sh(|E_{n+1}|^X))$$ 
is again a weak equivalence of simplicial sets. Hence $\Sing |E|_n \to \Omega \Sing |E|_{n+1}$ is a weak equivalence in $\sPrep$. 
\end{proof}
\begin{prop}\label{stablesingpretop}
Let $E$ be a symmetric spectrum and let $X$ be a complex manifold. There is a natural isomorphism 
$$\Hom_{\hoSpsPre}(\Sigma^{\infty}_+(X), \Sing |E|) \cong \Hom_{\hosSpt}(\Sigma^{\infty}_+(X), |E|).$$
\end{prop}
\begin{proof}
By Lemma \ref{stablesingfibrant}, $\Sing |E|$ is fibrant in the stable local projective model structure on $\SpsPre$. Since $\Sigma^{\infty}_+(X)$ is cofibrant, there is a sequence of natural bijections
$$\Hom_{\hoSpsPre}(\Sigma^{\infty}_+(X), \Sing |E|) \cong \pi_0(\Sh(|E|^X)) \cong \Hom_{\hosSpt}(\Sigma^{\infty}_+(X), |E|).$$
\end{proof}
Propositions \ref{stableconstant} and \ref{stablesingpretop} show that the constant presheaf of symmetric spectra given by a symmetric spectrum $E$ represents $E$-cohomology for manifolds also in $\hoSpsPre$. 
\begin{cor}\label{stablecorconstant}
Let $E$ be a symmetric spectrum and let $X$ be a complex manifold. There is a natural isomorphism 
\[
\Hom_{\hoSpsPre}(\Sigma^{\infty}_+(X), E) \cong \Hom_{\hosSpt}(\Sigma^{\infty}_+(X), |E|).
\]
\end{cor}
%

%%%%%%%%%%%%%%%%%%%%%%%%%%%%%%%%%%%%%%
%
%%%%%%%%%%%%%%%%%%%%%%%%%%%%%%%%%%%%%%
%%%%%%%%%%%%%%%%%%%%%%%%%%%%%%%%%%%%%%
%%%%%%%%%%%%%%%%%%%%%%%%%%%%%%%%%%%%%%
%

\section{Hodge filtered cohomology groups}

\subsection{The definition}

In this section we work on the site $\mathbf{T}=\Manc$ of complex manifolds with the topology defined in Section \ref{section2.1}. 

\begin{defn}
A {\em rationally even spectrum} is a symmetric spectrum $E$ with the property that $\pi_jE\otimes \Q=0$ if $j$ is odd.
\end{defn}

Let $E\in \sSp$ be a rationally even spectrum and let
\[
\iota \colon E \to E \wedge H\C=:E_{\C}
\]
be a map in $\sSp$ which induces for every $n$ the map 
\[
\pi_{2n}(E) \xrightarrow{(2\pi i)^n} \pi_{2n}(E_{\C})
\]
defined by multiplication by $(2\pi i)^n$ on homotopy groups. 

Let $p$ be an integer. Multiplication by $(2\pi i)^p$ on homotopy groups determines a map  
\[
E \xrightarrow{(2\pi i)^p\iota} E \wedge H\C.
\]
Let  
\[
E\wedge H\C \to H(\pi_{2*}E\otimes \C)
\]
be a map in $\sSp$ that induces the isomorphism  
\[
\pi_{2*}(E\wedge H\C) \cong \pi_{2*}E \otimes \C.
\]
Let 
\begin{equation}
\label{tauE}
\tau^E \colon E \to H(\pi_{2*}E\otimes \C)
\end{equation}
be the composition considered as a map in $\SpsPre$. \\
Composition with the canonical map 
\[
H(\pi_{2*}E\otimes \C) \to H(\Omega^*(\pi_{2*}E\otimes \C))
\]
yields a map of presheaves of symmetric spectra on $\Manc$ 
\begin{equation}\label{vertical}
E \to H(\Omega^*(\pi_{2*}E\otimes \C)).
\end{equation}

Given an integer $p$, we define the presheaf of symmetric spectra $\ED(p)$ by the homotopy cartesian square in $\SpsPre$ 
\begin{equation}\label{4.12E}
\xymatrix{
\ED(p) \ar[d] \ar[r] & E \ar[d] \\
H(\Omega^{*\geq p}(\pi_{2*}E\otimes \C)) \ar[r] & H(\Omega^*(\pi_{2*}E\otimes \C)).}
\end{equation}
%
%In order to define the presheaf of symmetric spectra $E(p)$ we have made choices for the construction of the map (\ref{vertical}). But any two choices of maps are homotopic in $\SpsPre$. % and hence the group $[\Xh,E(p)]$ does not depend on the choices we made. 
%This allows the following definition.

\begin{defn}\label{def4.34}
Let $n$ and $p$ be integers, $\Xh$ be a presheaf of symmetric spectra on $\Manc$, and $E$ be a  rationally even spectrum. The \emph{Hodge filtered $E$-cohomology groups} $E^n_{\Dh}(p)(\Xh)$ are defined as
\[
E^n_{\Dh}(p)(\Xh):=\Hom_{\hoSpsPre}(\Xh, \Sigma^n \ED(p)).
\]
In particular, for a complex manifold $X$, we have
\[
E^n_{\Dh}(p)(X)= \Hom_{\hoSpsPre}(\Sigma_+^{\infty}(X), \Sigma^n \ED(p)).
\]
\end{defn}

\begin{remark}\label{remsplitting}
It follows from the definition of $\ED(p)$ and the splitting of ordinary cohomology theories that rationally Hodge filtered complex cohomology splits in the same way into a sum of rational Deligne cohomology groups as rational $E$-cohomology splits into a sum of ordinary cohomology. More precisely, the definitions are chosen such that there is an isomorphism
%\begin{equation}\label{MUDsplitoverQ0}
%\MUD^0(0)X \cong \bigoplus_n H_{\Dh}^{2n}(X;\Q(n))\otimes \pi_{2n}MU
%\end{equation}
%and, more generally,
\begin{equation}\label{MUDsplitoverQp}
(E\wedge H\Q)_{\Dh}^{n}(p)X  \cong \bigoplus_j H_{\Dh}^{n+2j}(X;\Q(p+j))\otimes \pi_{2j}E.
\end{equation}
\end{remark}

\begin{remark}
For $p=0$, the Deligne complex $\Z_{\Dh}(0)$ is canonically isomorphic to $\Z$ and the cohomology groups $H^n_{\Dh}(X;\Z(0))$ are isomorphic to integral cohomology $H^n(X;\Z)$. But for generalized Hodge filtered $E$-cohomology, it is in general not true that we recover ordinary $E$-cohomology groups if $p=0$. The bottom row in the diagram
\begin{equation}\label{diagramE0}
\xymatrix{
\ED(0) \ar[d] \ar[r] & E \ar[d] \\
H(\Omega^{*\geq 0}(\pi_{2*}MU\otimes \C)) \ar[r] & H(\Omega^*(\pi_{2*}E\otimes \C))}
\end{equation}
is not an equivalence. By definition of 
\[
H(\Omega^{*\geq 0}(\pi_{2*}E\otimes \C)) = \bigvee_j \Sigma^{2j} H(\Omega^{*\geq j}(\pi_{2j}E\otimes \C))
\]
the step of the Hodge filtration that we divide out depends on the grading of $\pi_{2*}E\otimes \C$. In particular, the top row in diagram (\ref{diagramE0}) defining $\ED(0)$ is not an equivalence. 
\end{remark}

\subsection{Functoriality and Mayer-Vietoris sequence}

Let $f \colon \Xh \to \Yh$ be a map of presheaves of symmetric spectra $f \colon \Xh \to \Yh$. Then it follows from definition that there is an induced pullback map 
\[
f^* \colon E^n_{\Dh}(p)(\Yh) \to E^n_{\Dh}(p)(\Xh).
\]
In particular, if $\rho \colon \Uh \to X$ is a hypercovering of a complex manifold $X$, then $\rho$ induces an isomorphism 
\begin{equation}\label{hyperiso}
\rho^* \colon E^n_{\Dh}(p)(X) \xrightarrow{\cong} E^n_{\Dh}(p)(\Uh).
\end{equation}

Furthermore, it is a general fact that when $\Zh$ is the homotopy pullback of a diagram in $\SpsPre$ 
\[
\xymatrix{
 & \Uh \ar[d] \\
\Vh \ar[r] & \Wh,}
\]
there is an induced long exact sequence 
\[
\ldots \to [\Xh, \Omega \Uh] \oplus [\Xh, \Omega \Vh] \to [\Xh, \Omega \Wh] \to [\Xh, \Zh] \to [\Xh, \Uh] \oplus [\Xh,\Vh] \to [\Xh,\Wh].
\]
Since $E(p)$ is defined as the homotopy pullback of (\ref{4.12E}), the Hodge filtered cohomology groups sit in a long exact sequence. For a complex manifold this sequence has the following form.

\begin{prop}\label{les}
For any complex manifold $X$, Hodge filtered $E$-cohomology groups sit in a long exact sequence  
\[
\begin{array}{rl}
\ldots \to H^{n-1}(X; \pi_{2*}E\otimes \C) & \to E^n_{\Dh}(p)(X) \to \\
\to E^n(X)\oplus H^n(X; \Omega^{*\geq p}(\pi_{2*}E\otimes \C)) & \to H^n(X; \pi_{2*}E\otimes \C) \to \ldots
\end{array}
\]
%where we use the abbreviation $\Vh_{2*}:=\pi_{2*}E\otimes \C$.
\end{prop}
\begin{proof}
By the above fact, it suffices to determine the groups in the long exact sequence resulting from the definition of $\ED(p)$ as the homotopy pullback of (\ref{4.12E}). This can be done using Proposition \ref{stablesingpretop} for $E^n(X)$ and using Proposition \ref{hypersymmgraded} for the hypercohomology groups. 
\end{proof}

\begin{remark}
If $X$ is a compact K\"ahler manifold, then the hypercohomology group of $\Omega^{*\geq p}(\pi_{2*}E\otimes \C)$ in the long exact sequence of Proposition \ref{les} detects the Hodge filtration, i.e. there is an isomorphism
\[
H^n(X; \Omega^{*\geq p}(\pi_{2*}E\otimes \C)) \cong \bigoplus_j F^{p+j}H^{n+2j}(X;\pi_{2j}E\otimes \C)=: F^{p+*}H^n(X;\pi_{2*}E\otimes \C).
\]
But for an arbitrary complex manifold $X$, the hypercohomology of $\Omega^{*\geq p}(\pi_{2*}E\otimes \C)$ is a much less well-behaved group. We will see below how this defect can be fixed for smooth complex algebraic varieties.
\end{remark}

For a cofibration of presheaves of symmetric spectra $i \colon \Yh \to \Xh$, for example the map induced by a monomorphism of simplicial presheaves, let $\Xh/\Yh$ be the quotient in $\SpsPre$. 
We define relative Hodge filtered $E$-cohomology groups to be
\begin{equation}\label{relativegroups}
E^n_{\Dh}(p)(\Xh,\Yh):= \Hom_{\hoSpsPre}(\Xh/\Yh, \Sigma^n \ED(p)).
\end{equation}

Just as in topology, one proves the following result. 
\begin{prop}\label{localization}
These relative groups sit in a long exact sequence 
\begin{equation}\label{hfiberseq}
\ldots \to E^{n-1}_{\Dh}(p)(\Xh,\Yh) \to E^n_{\Dh}(p)(\Xh) \to E^n_{\Dh}(p)(\Yh) \to E^n_{\Dh}(p)(\Xh,\Yh) \to \ldots 
\end{equation}
induced by the cofiber sequence $\Yh \to \Xh \to \Yh/\Xh$.
\end{prop}

Moreover, there is a Mayer-Vietoris type long exact sequence for Hodge filtered cohomology groups. 
\begin{prop}\label{mvseq}
Let $X$ be a complex manifold and $U$ and $V$ be two open submanifolds such that $U \cup V = X$. Then there is a long exact Mayer-Vietoris sequence 
\[
\ldots \to E^n_{\Dh}(p)(X) \to E^n_{\Dh}(p)(U) \oplus E^n_{\Dh}(p)(V) \to E^n_{\Dh}(p)(U\cap V) \to E^{n+1}_{\Dh}(p)(X) \to \ldots
\]
\end{prop}
\begin{proof}
The long exact sequence is induced by applying the functor 
\[
\Hom_{\hoSpsPre}(\Sigma^{\infty}_+(-), \Sigma^n \ED(p))
\]
to the pushout diagram 
\[
\xymatrix{
U\cap V \ar[d] \ar[r] & U \ar[d] \\
V \ar[r] & X}
\]
of simplicial presheaves corresponding to the covering of $X$ by $U$ and $V$. 
\end{proof}

Finally, let $g \colon E \to F$ be a morphism of rationally even spectra and $\tau^F \colon F \to H(\pi_{2*}F\otimes \C)$ be a map with the properties described at the beginning of the previous section (for $\tau^E$). Then we can choose  $\tau^E \colon E\to H(\pi_{2*}E\otimes \C)$ with the properties we required above such that we get a commutative diagram  
\[
\xymatrix{
E \ar[r] \ar[d]_{\tau^E} & F \ar[d]^{\tau^F} \\
H(\pi_{2*}E\otimes \C) \ar[r] & H(\pi_{2*}F\otimes \C).}
\]
Together with the induced map on filtered forms we obtain a commutative diagram of presheaves of symmetric spectra 
\begin{equation}\label{spectramapdiag}
\xymatrix{
E \ar[d] \ar[r] & H(\Omega^{*}(\pi_{2*}E\otimes \C)) \ar[d] & \ar[l] 
H(\Omega^{*\geq p}(\pi_{2*}E\otimes \C))  \ar[d] \\
F \ar[r] & H(\Omega^{*}(\pi_{2*}F\otimes \C)) & \ar[l] 
H(\Omega^{*\geq p}(\pi_{2*}F\otimes \C)).}
\end{equation}
This diagram induces a map of homotopy pullbacks 
\[
g(p) \colon \ED(p) \to \FD(p)
\]
and hence map of Hodge filtered cohomology groups  
\[
g_{\Dh}(p) \colon E^n_{\Dh}(p)(\Xh) \to F^n_{\Dh}(p)(\Xh)
\]
for every presheaf of symmetric spectra $\Xh$.

\begin{prop}\label{4.13}
Let $g \colon E \to F$ be a stable equivalence of spectra and let $\tau^F \colon F \to H(\pi_{2*}F\otimes \C)$ be a map with the properties described above. Then the map 
\[
g(p) \colon \ED(p) \to \FD(p)
\]
is a stable equivalence in $\SpsPre$. In particular, the induced homomorphism 
\[
g_{\Dh}(p) \colon E^n_{\Dh}(p)(\Xh) \to F^n_{\Dh}(p)(\Xh)
\]
is an isomorphism for every $\Xh \in \SpsPre$. 
\end{prop}
\begin{proof}
Since $g$ is a homotopy equivalence, the vertical maps of diagram (\ref{spectramapdiag}) are pointwise equivalences and hence stable equivalences in $\SpsPre$. Hence the induced map of homotopy pullbacks is a stable equivalence as well. 
\end{proof}
\begin{remark}
The definition of Hodge filtered $E$-cohomology groups involves the
choice of a map 
\[
\tau^{E} \colon E\wedge H\C\to H(\pi_{2\ast}E\otimes\C).
\]
By our assumption on $E$, the space of such choices is simply
connected.     This implies that the Hodge filtered $E$-cohomology
groups of Definition \ref{def4.34} are independent of this choice, and a map 
$g \colon E\to F$ of rationally even spectra induces a well defined map from
Hodge filtered $E$-cohomology groups to Hodge filtered $F$ cohomology groups.
\end{remark}

\subsection{The fundamental short exact sequence for compact K\"ahler manifolds}

From now on we assume that $X$ is a compact K\"ahler manifold. Let $p$ be an integer and let $E$ be a rationally even spectrum.  
We will show that the group $\ED^{2p}(p)X$ fits into a short exact sequence that is a natural generalization of the short exact sequence for Deligne cohomology (\ref{sesDeligneintro}). The argument is analogous to the case of Deligne cohomology for which we refer to \cite[\S 12]{voisinbook}.

Therefore, we split the long exact sequence of Proposition \ref{les} into a short exact sequence. For a compact K\"ahler manifold $X$, the Hodge filtration yields an isomorphism  
\[
H^{n+2j}(X; \Omega^{*\geq p+j}(\pi_{2j}E\otimes \C)) \cong F^{p+j}H^{n+2j}(X; \pi_{2j}E \otimes \C).
\]
There is an analogue of the intermediate Jacobian defined which we will denote by $J^{2p-1}_{E}(X)$. It arises as the cokernel of the map
\begin{equation}\label{cokernel}
E^{2p-1}X \oplus F^{p+*}H^{2p-1}(X; \pi_{2*}E \otimes \C) \to H^{2p-1}(X; \pi_{2*}E \otimes \C)
\end{equation}
where we use the notation
\[
F^{p+*}H^{2p-1}(X; \pi_{2*}E \otimes \C):=\bigoplus_j F^{p+j}H^{2p-1+2j}(X; \pi_{2j}E \otimes \C).
\]
For every $j$, the Hodge filtration yields a decomposition as a direct sum
\[
H^{2p-1+2j}(X; \pi_{2j}E \otimes \C)\cong F^{p+j}H^{2p-1+2j}(X; \pi_{2j}E \otimes \C) \oplus \overline{F^{p+j}H^{2p-1+2j}(X; \pi_{2j}E \otimes \C)}
\]
where the bar denotes the image under complex conjugation.

Since this is a direct sum decomposition, the intersection
\[
F^{p+*}H^{2p-1}(X; \pi_{2*}E \otimes \C) \cap H^{2p-1}(X; \pi_{2*}E \otimes \R) =\{0\}
\]
is trivial. Considering $F^{p+*}H^{2p-1}(X; \pi_{2*}E \otimes \C)$ as a subspace of $E^{2p-1}(X)\otimes \C$ we see that the map
\[
E^{2p-1}(X)\otimes \R \to E^{2p-1}(X)\otimes \C/F^{p+*}H^{2p-1}(X; \pi_{2*}E \otimes \C)
\]
%$$H^*(X; \pi_{2*}E \otimes \R)^{2p-1} \to H^*(X; \pi_{2*}E \otimes \C)^{2p-1}/F^{p+*}H^*(X; \pi_{2*}E \otimes \C)^{2p-1}$$
is an isomorphism of $\R$-vector spaces. Therefore, the lattice 
\[
E^{2p-1}(X) \subset E^{2p-1}(X)\otimes \R
\]
is a lattice in the $\C$-vector space 
\[
E^{2p-1}(X)\otimes \C/F^{p+*}H^*(X; \pi_{2*}E \otimes \C)^{2p-1}.
\]

\begin{defn}
We define the $p$th generalized Jacobian of $X$ to be the compact complex torus  
\[
J_E^{2p-1}(X):= E^{2p-1}(X)\otimes \C/(F^{p+*}H^*(X; \pi_{2*}E \otimes \C)^{2p-1} + E^{2p-1}(X)).
\]
\end{defn}

%
%$$E^{2p-1}X \oplus \bigoplus_{j} F^{p+j}H^{2(p+j)-1}(X; \pi_{2j}E \otimes \C) \to \bigoplus_{j \geq 0} H^{2(p+j)-1}(X; \pi_{2j}E \otimes \C).$$
%
%Using the Hodge decomposition, we obtain that this map splits into
%$$\begin{array}{c}
%E^{2p-1}X \oplus \bigoplus_{j} \left(\bigoplus_{s\geq p, s+t=2(p+j)-1} H^{s,t}(X;\C)\otimes \pi_{2j}E\right) \to \\
%\to  \bigoplus_{j} \left(\bigoplus_{s+t=2(p+j)-1} H^{s,t}(X;\C) \otimes \pi_{2j}E \right).
%\end{array}$$
%
%Hence the cokernel of (\ref{cokernel}) is the quotient of 
%$$E^{2p-1}X\otimes \C \cong H^{*}(X; \pi_{2*}E \otimes \C)^{2p-1}$$ 
%by the (integral) group $E^{2p-1}X$ and exactly half of the Hodge decomposition of the cohomology in each degree. 

\begin{remark}\label{remrealliegroup}
The $p$th generalized Jacobian of $X$ is isomorphic to the group
$E^{2p-1}(X)\otimes \R/\Z$. 
% where the latter group inherits its complex structure from a complex structure on $\R/\Z$. 
In particular, $J_E^{2p-1}(X)$, as a real Lie group, is a homotopy
invariant of $X$. Though as a complex Lie group it depends on the
complex structure of $X$.   Since $X$ is compact,
there is an isomorhism $E^{\ast}(X)\otimes \Q\approx
H^{\ast}(X;\Q)\otimes \pi_{\ast}E$.  When $E$ has a homotopy
associative multiplication, this means that $\oplus_{p} J^{2p-1}$ is a
{\em flat} $\pi_{\ast}E$-module.
%$$J_{E}^{2p-1}(X) \cong E^{2p-1}(X)\otimes \R/\Z.$$ 
%
\end{remark}

To complete the picture, let $\HdgE^{2p}(X)$ be the subgroup of $E^{2p}(X)$ that is given as the pullback
\begin{equation}\label{foot}
\xymatrix{
\HdgE^{2p}(X) \ar[d] \ar[r] & E^{2p}(X) \ar[d] \\ 
\bigoplus_j F^{p+j}H^{2p+2j}(X; \pi_{2j}E \otimes \C) \ar[r] & \bigoplus_jH^{2p+2j}(X; \pi_{2j}E \otimes \C).}
\end{equation}
The group $\HdgE^{2p}(X)$ is determined by the Hodge structure on the cohomology of the K\"ahler manifold $X$. The Hodge decomposition of the lower right corner of (\ref{foot}) yields a splitting as a direct sum 
\[
H^{*}(X; \pi_{2*}E\otimes \C)^{2p}= \bigoplus_{j \geq 0}\left(\bigoplus_{s+t=2(p+j)} H^{s,t}(X;\C)\otimes \pi_{2j}E\right).
\]
By Hodge symmetry, we see that $\HdgE^{2p}(X)$ is the subgroup of $E^{2p}(X)$ consisting of those elements whose images under the map 
\[
E^{2p}(X) \to \bigoplus_{j}\left(\bigoplus_{s+t=2(p+j)} H^{s,t}(X;\C)\otimes \pi_{2j}E\right)
\]
lie in the groups $H^{p+j, p+j}(X;\C)\otimes \pi_{2j}E$ for $j \in \Z$.

Summarizing we obtain the following theorem whose second assertion follows immediately from Theorem \ref{D.6}. 
\begin{theorem}\label{thmfundses}
Let $p$ be an integer, $X$ be a compact K\"ahler manifold and let $E$ be a rationally even spectrum. The Hodge filtered complex cohomology groups $\ED^{2p}(p)(X)$ fit into the short exact sequence
\begin{equation}\label{ses}
0\to J_E^{2p-1}(X) \to \ED^{2p}(p)(X) \to \HdgE^{2p}(X)\to 0
\end{equation}
where $\HdgE^{2p}(X)$ is the pullback defined by (\ref{foot}).

Moreover, if there is a map $E \to H\Z$ of symmetric spectra, we obtain an induced map of short exact sequences 
\begin{equation}\label{mapofses}
\xymatrix{
0 \ar[r] & J_E^{2p-1}(X) \ar[d] \ar[r] & \ED^{2p}(p)X \ar[d] \ar[r] & \HdgE^{2p}X \ar[d]\ar[r]&0\\
0 \ar[r] & J^{2p-1}(X) \ar[r] & H_{\Dh}^{2p}(X;\Z(p)) \ar[r] & \Hdg^{2p}(X;\Z) \ar[r] & 0} 
\end{equation}
\end{theorem}

%%%%%%%%%%%%%%%%%%%%%%%%%%%%%%

\section{Hodge filtered complex bordism}

%\subsection{Product structures}

In this section we focus on the case of Hodge filtered complex cohomology groups defined via the Thom spectrum $MU$ of complex bordism. We recall that the ring $\pi_{2*}MU$ is isomorphic as a graded ring to the polynomial ring $\Z[x_2, x_4, \ldots]$ where for each $i\geq 0$, $x_{2i}$ is a generator in degree $2i$. 

Before we apply the construction let us have a closer look at the choice of the map
\[
MU \to H(\pi_{2*}MU\otimes \C).
\]
Although any two choices give the same cohomology groups, we have to take a little more care in order to get a multiplicative structure. We explain this point only for the example of $MU$ and leave it to the reader to pick his favorite spectrum $E$ and to do the necessary adjustments.

For $E=MU$, the three corners of the diagram corresponding to (\ref{4.12E}) that would define the Hodge filtered cobordism spectrum as a homotopy pullback all carry a natural product structure. These structures will induce a natural product structure on Hodge filtered cobordism groups, once we make sure that the maps are multiplicative, i.e. preserve these product structures. By the construction of monoidal Eilenberg-MacLane spectra, the map  
\[
H(\Omega^{*\geq p}(\pi_{2*}MU\otimes \C)) \to H(\Omega^*(\pi_{2*}MU\otimes \C))
\]
induced by the canonical inclusion of complexes is a morphism of monoids in $\SpsPre$.

For the other map we make use of the fact that the rationalized
spectrum $MU_{\Q}\approx H\Q\wedge MU$ is a free commutative algebra
over $H\Q$, and so the commutative $S$-algebra maps $MU\to E$ to any
rational commutative $\Sph$-algebra $E$ are in one to one correspondence
with the ring homomorphisms $\pi_{\ast}MU\to \pi_{\ast}E$.  The
correspondence is the one associating to a map its effect on homotopy
groups.  We can therefore specify, up to a choice of point in a simply
connected space, a map of commutative $\Sph$-algebras
\[
\iota \colon MU\to H(\pi_{\ast}MU\otimes\C),
\]
by requiring that for every $n$ the induced map 
\[
\pi_{2n}\iota \colon \pi_{2n}MU\to \pi_{2n}(H(\pi_{\ast}MU\otimes\C)) = \pi_{2n}MU\otimes \C
\]
is multiplication by $(2\pi i)^{n}$.  We similarly obtain a map of commutative $\Sph$-algebras 
\[
\tau':= \bigvee_{p\in \Z} (\tau')^{p} \colon \bigvee_{p\in\Z} MU \to \bigvee_{p\in\Z} 
H(\pi_{2\ast}MU\otimes \C)
\]
in which $\pi_{\ast}(\tau')^{p} = (2\pi i)^{p}\pi_{\ast}\iota$.   Composing with the de Rham equivalence 
\[
\bigvee_{p\in\Z}  H(\pi_{2\ast}MU\otimes \C) \xrightarrow{\approx} \bigvee_{p\in\Z} 
H(\Omega(\pi_{2\ast}MU\otimes C))
\]
gives a map
\[
\tau:= \bigvee_{p\in \Z} \tau^{p} \colon \bigvee_{p\in\Z} MU \to \bigvee_{p\in\Z} 
H(\Omega(\pi_{2\ast}MU\otimes C))
\]
of commutative $\Sph$-algebras in {\em presheaves of symmetric spectra}, whose stalkwise effect
on homotopy groups is 
\[
\pi_{\ast}\tau^{p} = (2\pi i)^{p}\pi_{\ast}\iota.
\]

We now define a presheaf of commutative $\Sph$-algebras 
\begin{equation}
\label{eq:2}
\bigvee_{p\in\Z}\MUD(p)
\end{equation}
by the homotopy cartesian square 
\newcommand{\mud}{MU_{\Dh}}
\begin{equation}
\label{eq:1}
\xymatrix{
*++!{\displaystyle\bigvee_{p\in\Z}\MUD(p)}  \ar[r]\ar[d]  &  *++!{\displaystyle\bigvee_{p\in\Z} MU} \ar[d]^{\tau}\\
*++!{\displaystyle\bigvee_{p\in\Z} H(\Omega^{\ast\ge p}(\pi_{2\ast}MU\otimes\C))}  \ar[r]       &
*++{\displaystyle\bigvee_{p\in\Z} H(\Omega^{\ast}(\pi_{2\ast}MU\otimes\C))},
}
\end{equation}
and for every presheaf of symmetric spectra $\Xh$, and $p\in\Z$ the {\em Hodge filtered complex bordism groups}
\[
\MUD^{n}(p) := \Hom_{\hoSpsPre}(\Xh, \Sigma^n \MUD(p)). 
\]  

This defines in particular, for every complex manifold $X$, Hodge filtered complex bordism groups $\mud^{n}(p)(X)$. Since \eqref{eq:2} is a commutative $\Sph$-algebra, we get the following multiplicative structure on Hodge filtered complex bordism groups of complex manifolds.
\begin{theorem}\label{products}
For every complex manifold $X$, there is a multiplicative structure on Hodge filtered complex bordism groups  
\[
\MUD^n(p)(X)\otimes \MUD^m(q)(X) \to \MUD^{n+m}(p+q)(X)
\]
induced by the structure of a commutative $\Sph$-algebra in $\SpsPre$ on \eqref{eq:2}. 
%$$\bigvee_{p\in \Z} \MUD(p).$$
%
The multiplication is graded-commutative in the sense that for $\alpha \in \MUD^n(p)(X)$, $\beta \in \MUD^m(q)(X)$ we have
\[
\alpha \beta = (-1)^{n+m}\beta \alpha \in \MUD^{n+m}(p+q)(X).
\]
Hence there is a structure of a graded-commutative ring on 
\[
\MUD^*(*)(X):=\bigoplus_{n,p} \MUD^n(p)(X).
\]
%Moreover, if $Y \to X$ is a holomorphic map of complex manifolds, the induced map of simplicial presheaves $X \to X/Y$ turns $\MUD^*(*)(X)$ into an $\MUD^*(*)(X,Y)$-module.
\end{theorem}

\begin{prop}
For every complex manifold $X$ and every $n$ and $p$, the map of spectra $MU\to H\Z$ induces a map from Hodge filtered complex bordism to Deligne cohomology groups of $X$
\[
\MUD^n(p)(X) \to H_{\Dh}^n(X; \Z(p))
\]
which respects the product structures on both sides.
\end{prop}
\begin{proof}
This is a consequence of Theorem \ref{D.6} and Theorem \ref{products}.
\end{proof}

\begin{prop}\label{les1}
For every complex manifold $X$, Hodge filtered complex bordism groups sit in long exact sequence  
%$$\begin{array}{rl}
%\ldots \to H^{n-1}(X; \Vh_{2*}) & \to \MUD(p)^{n}(X) \to \\
%\to MU^n(X)\oplus H^n(X; \Omega^{*\geq p}(\Vh_{2*})) & \to H^n(X; \Vh_{2*}) \to \ldots
%\end{array}$$ 
%and
\[
\begin{array}{rl}
\ldots \to MU^{n-1}(X)\otimes \C & \to \MUD^n(p)(X) \to \\
\to MU^n(X)\oplus H^n(X; \Omega^{*\geq p}(\pi_{2*}MU\otimes \C)) & \to MU^n(X)\otimes \C \to \ldots.
\end{array}
\]
\end{prop}
\begin{proof}
This follows again from the definition of $\MUD(p)$ as the homotopy pullback in diagram \eqref{eq:1} and the isomorphism 
\[
MU^n(X)\otimes \C \cong H^n(X;\pi_{2*}MU\otimes \C).
\]
\end{proof}
%
%%%%%%%%%%%%%%

\begin{example}%subsection{Hodge filtered bordism groups of a point}

The Hodge filtered cobordism groups of a point can be read off from the long exact sequence of Proposition \ref{les1}. We split sequence (\ref{les1}) into 
\[
0\to \mathrm{coker}(\alpha) \to \MUD^{n}(p)(\pt) \to \mathrm{ker}(\beta) \to 0
\]
where $\alpha$ is the map
\[
\alpha \colon MU^{n-1} \oplus H^{n-1}(\pt; \Omega^{*\geq p}(\pi_{2*}MU \otimes \C)) \to MU^{n-1}\otimes \C\]
and $\beta$ is the map
\[
\beta \colon MU^{n} \oplus H^n(\pt; \Omega^{*\geq p}(\pi_{2*}MU \otimes \C)) \to MU^{n}\otimes \C.
\]
First let $n=2k$ be even. In this case, the group $MU^{2k-1}\otimes \C$ vanishes. Hence the cokernel of $\alpha$ is trivial and $\MUD^{n}(p)(\pt)$ is isomorphic to $\mathrm{ker}(\beta)$. The groups $H^{n+2j}(\pt; \Omega^{*\geq p+j}(\pi_{2j}MU \otimes\C))$ vanish if $n+2j \neq 0$ or if $p+j >0$. Hence if $p>0$ or $k<p$, we see that $\MUD^{2k}(p)(\pt)$ is isomorphic to the kernel of the inclusion $MU^{2k}\to MU^{2k}\otimes \C$ and hence vanishes.

If $p\leq 0$, then $H^{2k}(\pt; \Omega^{*\geq p}(\pi_{2*}MU \otimes\C))$ is isomorphic to $MU^{2k}\otimes \C$ for $k \geq p$ and vanishes otherwise. Hence, for $k \geq p$, $\MUD^{2k}(p)(\pt)$ is isomorphic to the kernel of the map 
\[
MU^{2k} \oplus MU^{2k}\otimes \C \to MU^{2k}\otimes \C
\]
and hence 
\[
\MUD^{2k}(p)(\pt)\cong MU^{2k}
\]
for $k\geq p$. In particular, this implies that the group $\MUD^{2k}(p)(\pt)$ vanishes for $k>0$. 

We remark that, for all $p\in \Z$, we have
\[
\MUD^{2p}(p)(\pt)\cong MU^{2p}.
\]

For $n=2k+1$ odd, the group $\mathrm{ker}(\beta)$ is trivial, since the source of $\beta$ vanishes.  So $\MUD^{2k+1}(p)(\pt)$ is isomorphic to the cokernel of $\alpha$. For $p>0$, using the identifications above, we get for all $k\in \Z$ 
\[
\MUD^{2k+1}(p)(\pt) \cong (MU^{2k}\otimes \C)/MU^{2k}\cong MU^{2k}\otimes \C/\Z.
\]
In particular, this implies that the group $\MUD^{2k+1}(p)(\pt)$ vanishes for $k>0$. 

For $p \leq 0$ and $k < p$, we still get 
\[
\MUD^{2k+1}(p)(\pt) \cong MU^{2k} \otimes \C/\Z.
\]

For $p\leq 0$ and $k\geq p$, we have 
\[
\MUD^{2k+1}(p)(\pt) \cong (MU^{2k} \otimes \C)/(MU^{2k}\oplus MU^{2k}\otimes \C) = 0.
\]
\end{example}
%
%
%%%%%%%%%%%%%%%%%%%%%%%%%%%%%%%%%
%

\section{Hodge filtered cohomology for complex algebraic varieties}

%In the final section we walk the first steps towards a cycle map for smooth complex algebraic varieties. After modifying our construction for smooth complex algebraic varieties following Beilinson's ideas \cite{beilinson}, we will show that Hodge filtered complex cobordism satisfies a projective bundle formula. This will lead to transfer morphisms for proper morphisms of algebraic varieties.  
%
We now turn to the case of algebraic varieties and, as described in the introduction, introduce a variation of our construction which takes into account mixed Hodge structures.
%
%For non-compact manifolds, the definition of Deligne cohomology that we used so far may produce infinite dimensional vector spaces. For algebraic varieties, there is a more sophisticated version that remedies this defect. 

\subsection{From the analytic to the Nisnevich topology}

To transfer the theory of Hodge filtered cohomology to algebraic varieties, we have to discuss the relationship between the analytic site and the algebraic site. We will use the term {\em complex variety} for a reduced separated scheme of finite type over $\C$. Let $\Smcnis$ be the site of smooth complex algebraic varieties with the Nisnevich topology. 

We recall that a distinguished square in $\Smcnis$ is a cartesian square of the form 
\begin{equation}\label{dsquare}
\xymatrix{U\times_XV\ar[r] \ar[d] & V \ar[d]^p \\
U \ar[r]^j & X}
\end{equation}
such that $p$ is an \'etale morphism, $j$ is an open embedding and $p^{-1}(X-U)\to X-U$ is an isomorphism, where the closed subsets are equipped with the reduced induced structure. A square of the form \eqref{dsquare} is an example of a covering in the Nisnevich topology. Moreover, the Nisnevich topology is generated by coverings of this form (see \cite[\S 3.1]{mv}).

For $X\in \Smcnis$, we denote by $X_{\an}$ the associated complex manifold.  
This defines a functor
\[
f^{-1} \colon \mathbf{Sm}_{\C} \to \Manc, ~ X \mapsto f^{-1}(X):= X_{\an}.
\]
Composition with $f^{-1}$ induces a functor 
\[
f_* \colon \SpsPreman \to \SpsPrenis
\]
of the corresponding categories of presheaves of symmetric spectra. It is the right adjoint of a functor 
\[
f^* \colon \SpsPrenis \to \SpsPreman.
\]
The pair of functors $(f^*,f_*)$ forms a Quillen pair of adjoints between the model categories of monoids in $\SpsPrenis$ and $\SpsPreman$ respectively. 

We let $Rf_*$ be the right derived functor. It is given by the composition $f^{-1}\circ R_{\an}$ of $f^{-1}$ with a fibrant replacement functor $R_{\an}$ in the model structure of monoids in $\SpsPreman$. Since representable objects are cofibrant, the adjointness implies the following fact. 
\begin{prop}\label{comparesites}
Let $X$ be a smooth complex algebraic variety and let $\Eh$ be a presheaf of symmetric spectra on the site $\Manc$. Then there is a natural isomorphism 
$$\Hom_{\hoSpsPreman}(\Sigma^{\infty}_+(X_{\an}), \Sigma^n \Eh) \cong \Hom_{\hoSpsPrenis}(\Sigma^{\infty}_+(X), \Sigma^n Rf_*\Eh).$$
\end{prop}
%\begin{proof}
%Since $\Sigma^{\infty}$ is left adjoint to taking the $0$-space of functorial fibrant replacement, it suffices to show the following statement. Let $\Yh$ be a simplicial presheaf on $\Manc$ which is fibrant in the injective model structure. Then we have to show that there is a natural isomorphism 
%$$\Hom_{\hosPre(\Manc)}(X_{\an}, \Yh) \cong \Hom_{\hosPre(\Smcnis)}(X, f_*\Yh).$$
%
%Since $\Yh$ is fibrant and $X$ and $X_{\an}$ are representable and hence cofibrant objects in $\Smcnis$ and $\Manc$ respectively, we have a commutative diagram of sets 
%$$\xymatrix{ 
%\Hom_{\hosPre(\Manc)}(X_{\an}, \Yh) \ar[d]_{\cong} \ar[r] & \Hom_{\hosPre(\Smcnis)}(X, f_*\Yh)\ar[d]^{\cong}\\
%\pi_0(\Yh(X_{\an})) \ar[r] & \pi_0((f_*\Yh)(X)).}
%$$
%By the definition of $f_*$, we have 
%$$\pi_0((f_*\Yh)(X))=\pi_0(\Yh(f^{-1}(X))=\pi_0(\Yh(X_{\an})).$$
%
%Hence the map in the lower row of the diagram is a bijection. Since the vertical maps are bijections, this shows that the map in the top row is a bijection as well. 
%\end{proof}
%We also denote by $f_*$ the induced functor from rings in the symmetric monoidal category $\SpsPreman$ to the category of rings in category $\SpsPrenis$. By abuse of notation, we will denote the derived functor from rings in $\SpsPreman$ to the category of rings in $\SpsPrenis$ by $Rf_*$ as well. It is given by composition $f^{-1}\circ R^{\mathrm{mo}}_{\an}$ of $f^{-1}$ with a fibrant replacement functor $R^{\mathrm{mo}}_{\an}$ in the model category of rings in $\SpsPreman$. \\
%
\begin{example}\label{derivedmonoid}
In particular, if $E\in \sSp$ is a symmetric spectrum,  we have a natural isomorphism 
\[
E^n(X_{\an}) \cong \Hom_{\hoSpsPrenis}(\Sigma^{\infty}_+(X), \Sigma^n Rf_*E).
\]
%
% e have shown that the corresponding presheaf of spectra still represents $E$-cohomology in the homotopy category of presheaves of symmetric spectra on $\Manc$, i.e. we have 
%$$E^n(M)\cong \Hom_{\hoSpsPre(\Manc)}(\Sigma^{\infty}_+(M), \Sigma^n E).$$ 
%
%Proposition \ref{comparesites} shows that $Rf_*E$ represents $E$-cohomology in $\hoSpsPrenis$ in the following sense. Given a smooth complex variety $X$,
\end{example}

\subsection{Differential forms with logarithmic poles}

For a smooth complex variety $X$, by Hironaka's theorem \cite{hironaka}, there exists a smooth complete variety $\oX$ over $\C$ with a closed embedding $j \colon X\into \oX$ such that $D:=\oX - X$ is a normal crossing divisor. % which is the union of smooth divisors. 
Let $\Omega^1_{\oX}\langle D \rangle$ be the locally free sub-module of $j_*\Omega^1_{X}$ generated by $\Omega^1_{X}$ and by $\frac{dz_i}{z_i}$ where $z_i$ is a local equation for an irreducible local component of $D$. The sheaf $\Omega^p_{\oX}\langle D \rangle$ of meromorphic $p$-forms on $\oX$ with at most logarithmic poles along $D$ is defined to be the locally free sub-sheaf $\bigwedge^p\Omega^1_{\oX}\langle D \rangle$ of $j_*\Omega_{\oX}^p$. The Hodge filtration $F$ on the complex of sheaves $\Omega^*_{\oX}\langle D \rangle$ is defined to be the filtration 
\[
F^p\Omega^*_{\oX}\langle D \rangle = \Omega^{* \geq p}_{\oX}\langle D \rangle.
\]
The filtration on the complex of logarithmic forms allows one to define a Hodge filtration on the complex cohomology of $X$ as
\begin{equation}\label{Hodgefiltration}
F^pH^n(X;\C):= \Imm (H^n(\oX; \Omega^{* \geq p}_{\oX}\langle D \rangle) \to H^n(\oX; \Omega^{*}_{\oX}\langle D \rangle)\cong H^n(X;\C)).
\end{equation}
It is a theorem of Deligne's \cite[Th\'eor\`eme 3.2.5 and Corollaire 3.2.13]{hodge2} that for smooth complex algebraic varieties, the homomorphism 
\[
H^n(\oX; \Omega^{* \geq p}_{\oX}\langle D \rangle) \to H^n(\oX; \Omega^{*}_{\oX}\langle D \rangle)
\]
is injective and the image is independent of the choice of $\oX$. In particular, the map \eqref{Hodgefiltration} induces an isomorphism
\begin{equation}\label{delignehodge}
H^n(\oX; \Omega^{* \geq p}_{\oX}\langle D \rangle) \cong F^pH^n(X;\C) \subset H^n(X;\C).
\end{equation}

%
%%%%%%%%%%%%%%%%%%%%%%%%%%%
%

Following Beilinson~\cite[\S 1.6]{beilinson} we now show that the Hodge filtration
$F^{p}H^{n}(X,\C)$ is represented by a map of
symmetric spectra.  More precisely, we construct a sequence of  complexes of
presheaves
\[
\cdots \to A^{p}_{\log} \to A^{p+1}_{\log} \to \cdots
\]
on $\Smcnis$ equipped with a weak equivalence
\[
\colim_{p\to\infty} HA^{p}_{\log} \simeq Rf_{\ast}H\C,
\]
and having the property that for all $X\in\Smc$, the map 
\[
\Hom_{\hoSpsPre(\Smcnis)}(\Sigma^{\infty}_{+}X, \Sigma^{n}HA_{\log}^{p}) \to 
\Hom_{\hoSpsPre(\Smcnis)}(\Sigma^{\infty}_{+}X, \Sigma^{n}Rf_{\ast}H\C)
\]
induces an isomorphism
\[
\Hom_{\hoSpsPre(\Smcnis)}(\Sigma^{\infty}_{+}X, \Sigma^{n}HA_{\log}^{p})\to F^{p}H^{n}(X,\C).
\]

We consider the category $\Smb$ whose objects are {\em smooth
compactifications}, i.e. pairs $X\subset \oX$ consisting of a
smooth variety $X$ embedded as an open subset of a smooth complete variety
$\oX$ and having the property that $\oX-X$ is a
divisor with normal crossings.   A map from $X\subset \oX$ to
$Y\subset \overline{Y}$ is a commutative diagram 
\[
\xymatrix{
X  \ar[r]\ar[d]  &  \oX \ar[d] \\
Y  \ar[r]        &\oY.
}
\]
%and a collection of maps $\{(U_{\alpha}\subset \overline{U}_{\alpha})\} \to (X\subset\overline{X})$ is a covering if the maps $U_{\alpha}\to X$ form a Nisnevich covering.   This defines a Grothendieck topology on $\Smb$. % (see also \cite[\S 2]{beilinson3}). 

%We denote by $\sPre(\Smb)$ the category of simplicial presheaves on $\Smb$  and let $\SpsPreSmb$ be the category of presheaves of symmetric spectra on $\Smb$. 

The complexes $\Omega^{\ast\ge p}_{\overline X}\langle D \rangle$ form a complex of presheaves $\Omega^{\ast\ge p}\langle D \rangle$ on $\Smb$.   Let 
\[
\Omega^{\ast\ge p}_{\oX}\langle D \rangle 
\to A^{p}_{\oX}\langle D \rangle 
\]
be any resolution by cohomologically trivial sheaves which is
functorial in $\overline{X}$, and let 
\[
\Omega^{\ast\ge p}\langle D \rangle 
\to A^{p}\langle D \rangle 
\]
be the associated map of complexes of presheaves.   For example, $A^{p}\langle D \rangle$
could be the Godemont resolution, or  the logarithmic Dolbeault
resolution (\cite[\S 8]{navarro}). %In the applications of the next sections we will pick the latter resolution.   
The $A^{p}\langle D \rangle$ are double complexes, though we will only consider their total complexes.   They may be chosen in such a way as to fit into a commutative diagram 
\[
\xymatrix{
\Omega^{*\geq p-1}\langle D \rangle  \ar[r]\ar[d]  &  \Omega^{*\geq p}\langle D \rangle \ar[d] \\
A^{p-1}\langle D \rangle  \ar[r]        & A^{p}\langle D \rangle
}
\]

We define the presheaf $A^p_{\log}$ of complexes on $\Smc$ by
\[
A^p_{\log} \colon X \mapsto \colim_{C(X)} A^p\langle D \rangle (\oX)
\]
where the colimit is taken over the directed category $C(X)$ of all smooth compactifications $\oX$ of $X$. Let $HA^{p}_{\log} \in \SpsPrenis$ be the associated Eilenberg-MacLane spectrum. 

For $X\in \Smc$, let $\pi(\Sigma^{\infty}_{+}X, \Sigma^{n}HA^{p}_{\log})$ 
denote the set of homotopy classes of maps. The Dold-Kan correspondence implies that there is a natural bijection
\[
\pi(\Sigma^{\infty}_{+}X, \Sigma^{n}HA^{p}_{\log}) \cong H^n(A^p_{\log}(X)).
\]
As in \cite[Theorem 4]{browngersten} it follows from Deligne's isomorphism \eqref{delignehodge} and the fact that  $A^p\langle D \rangle$ is pseudo-flasque that we have a natural isomorphism
\[
H^n(A^p \langle D \rangle (\oX)) \cong F^pH^n(X;\C).
\]
Since the colimit defining $A^p_{\log}(X)$ is filtered, this implies that the natural map 
\[
H^n(A^p_{\log}(X)) \to F^pH^n(X;\C)
\]
is an isomorphism too. Hence we obtain a natural isomorphism
\[
\pi(\Sigma^{\infty}_{+}X, \Sigma^{n}HA^{p}_{\log}) \cong F^pH^n(X;\C).
\]
The functor $X\mapsto H^*(X;\C)$ satisfies descent for the Nisnevich topology in the sense that every distinguished square \eqref{dsquare} induces a long exact Mayer-Vietoris sequence. By \cite[Th\'eor\`eme 1.2.10 and Corollaire 3.2.13]{hodge2}, the Hodge filtration respects this sequence and yields a long exact sequence 
\[
\ldots \to F^pH^{q-1}(X;\C) \to F^pH^q(U\times_X V;\C) \to F^pH^q(U;\C) \oplus F^pH^q(V;\C) \to F^pH^q(X;\C) \to \ldots
\] 
%Thus the Mayer-Vietoris sequence for the groups $F^pH^n(X;\C)$ shows that the groups $\pi(\Sigma^{\infty}_{+}X, \Sigma^{n}HA^{p}_{\log})$ satisfy Nisnevich descent. 
This allows us to deduce from the Verdier hypercovering theorem that the natural map 
\[
\pi(\Sigma^{\infty}_{+}X, \Sigma^{n}HA^{p}_{\log}) \to 
\Hom_{\hoSpsPrenis}(\Sigma^{\infty}_{+}X, \Sigma^{n}HA^{p}_{\log})
\]
is a bijection (see \cite[Theorem 3.5]{compact} for more details). This implies the following result. 

\begin{prop}
\label{compind}
For every $X$, and every $n\ge 0$, there is a natural isomorphism 
\[ 
\Hom_{\hoSpsPrenis}(\Sigma^{\infty}_{+}X, \Sigma^{n}HA_{\log}^{p}) \cong F^{p}H^{n}(X;\C).
\] 
\end{prop}

%
%
%%%%%%%%%%%%%%%%%%%%%%%%%%%%%%%%%%%%
%
%%%%%%%%%%%%%%%%%%%%%%%%%%%%%%%%%%%
%
\subsection{Generalized Deligne-Beilinson cohomology for smooth complex algebraic varieties}
Let $A^{\ast} = \bigoplus A^{p,q}$ be the presheaf given by the functorial resolution of $\Omega^*$ compatible with the resolution of $\Omega^{\ge p}\langle D \rangle$ chosen in the previous subsection. For an evenly graded $\C$-algebra $\Vh_{2*}$, let $HA^*(\Vh_{2*})$ denote the presheaf of commutative $\Sph$-algebras 
\[
HA^*(\Vh_{2*}) = \bigvee_j \Sigma^{2j}HA^{*}(\Vh_{2j}) = \bigvee_j \Sigma^{2j} HA^{*} \otimes \Vh_{2j}.
\]
Moreover, let $HA^{p+*}_{\log}(\Vh_{2*})$ be the presheaf of commutative $\Sph$-algebras  
\[
HA^{p+*}_{\log}(\Vh_{2*}) = \bigvee_j \Sigma^{2j}HA_{\log}^{p+j}(\Vh_{2j}),
\]
where we denote
\[
HA_{\log}^{p+j} (\Vh_{2j}) = HA^{p+j}_{\log} \otimes \Vh_{2j}.
\]
Any choice of resolution $A^p\langle D \rangle$ of the complexes of logarithmic forms is equipped with a natural morphism of presheaves of symmetric spectra
\begin{equation}
\label{HAlogmap}
HA^{p+*}_{\log}(\Vh_{2*}) \to Rf_*HA^*(\Vh_{2*}).
\end{equation}

Let $E$ be a rationally even spectrum in $\sSp$. Let  
\[
\tau^E \colon E \to H(\pi_{2*}E\otimes \C)
\]
be the map \eqref{tauE}. 
It induces a map 
\[
Rf_*E \to Rf_*H(\pi_{2*}E\otimes \C)
\]
in $\SpsPrenis$. Composition with the maps
\[
Rf_*H(\pi_{2*}E\otimes \C) \to Rf_*H(A^*(\pi_{2*}E\otimes \C))
\]
defines a map in $\SpsPrenis$  
\[
Rf_*E \to Rf_*H(A^*(\pi_{2*}E\otimes \C)).
\]
We define the presheaf of symmetric spectra $E_{\log}(p)$ by the homotopy cartesian square   
\begin{equation}\label{diagramEDB}
\xymatrix{
E_{\log}(p) \ar[d] \ar[r] & Rf_*E \ar[d] \\
H(A^{p+*}_{\log}(\pi_{2*}E\otimes \C)) \ar[r] & Rf_*H(A^*(\pi_{2*}E\otimes \C)).}
\end{equation}
%of rings in $\SpsPre$ on the site $\Smcnis$.

\begin{defn}\label{EDBrefined}
For a presheaf of symmetric spectra $\Xh$ on $\Smcnis$, we define the logarithmic Hodge filtered $E$-cohomology groups to be 
\[
\EDB(p)^n(\Xh) := \Hom_{\hoSpsPre}(\Xh, \Sigma^n E_{\log}(p)).
\]
\end{defn}

For a smooth complex algebraic variety $X$, we denote the $E$-cohomology group $E^n(X_{\an})$ of the associated complex manifold $X_{\an}$ by $E^n(X)$. Moreover, we use the notation
\[
F^{p+*}H^*(X; \pi_{2*}E\otimes \C)^{n} = \bigoplus_{j} F^{p+j}H^{n+2j}(X; \pi_{2j}E\otimes \C)
\]
for the sum of Hodge filtered cohomology groups.

\begin{prop}\label{Elesalg}
Hodge filtered cohomology groups sit in long exact sequences  
\[
\begin{array}{rl}
\ldots \to H^*(X; \pi_{2*}E\otimes \C)^{n-1} & \to \Elog^n(p)(X) \to \\
\to E^n(X)\oplus F^{p+*}H^*(X; \pi_{2*}E\otimes \C)^n & \to H^*(X; \pi_{2*}E\otimes \C)^n \to \ldots
\end{array}
\]
and
\[
\begin{array}{rl}
\ldots \to E^{n-1}(X)\otimes \C & \to \Elog^n(p)(X) \to \\
\to E^n(X)\oplus F^{p+*}H^*(X; \pi_{2*}E\otimes \C)^n & \to E^n(X)\otimes \C \to \ldots.
\end{array}
\]
%where we denote $\Vh_{2*}:=\pi_{2*}E\otimes \C$.  
\end{prop}
\begin{proof}
As before, this is a consequence of the definition of $\Elog(p)$ as a homotopy pullback. The new ingredients are the identifications of $E$-cohomology groups in Proposition \ref{comparesites} and of the Hodge filtered cohomology groups in \eqref{compind}. 
\end{proof}

%\begin{remark}
%As discussed for $E=MU$, one can show that there is a ring structure on
%$$\EDB^*(*)(X) = \bigoplus_{n,p} \EDB^n(p)(X)$$
%for any complex manifold $X$. 
%\end{remark}

\begin{remark}
If $X$ is a smooth projective complex variety, the new Hodge filtered cohomology groups $\EDB^n(p)(X)$ of Definition \ref{EDBrefined} are canonically isomorphic to the groups $\ED^n(p)(X_{\an})$ of Definition \ref{def4.34}. 
\end{remark}

%\begin{remark}
%For the definition of $E$-cohomology groups in Definition \ref{EDBrefined}, it is not necessary that $E$ is a ring spectrum. If $E\in \sSp$ is any symmetric spectrum such that $\pi_jE\otimes \C=0$ if $j$ is odd, then one can define Hodge filtered $E$-cohomology groups for any smooth complex variety just as in Definition \ref{EDBrefined} by using the non-monoidal versions of $H$ and $Rf_*$. To avoid too many confusing changes of the playing field, we decided to stick to the multiplicative formulation.
%\end{remark}

\subsection{$\A^1$-homotopy invariance}

\begin{theorem}\label{extendeda1}
Let $E$ be a rationally even ring spectrum and let $X$ be a smooth variety over $\C$. The projection $\pi \colon X\times \A^1 \to X$ induces an isomorphism
\[
\pi^* \colon \Elog^n(p)(X) \xrightarrow{\cong} \Elog^n(p)(X\times \A^1).
\]
%
%$V\to X$ a vector bundle over $X$, and let $\pi:W\to X$ be a $V$-torsor. Then the pullback 
%$$\pi^*:\EDB^n(p)(X) \stackrel{\cong}{\to} \EDB^n(p)(W)$$ 
%is an isomorphism for every $n$ and $p$. 
\end{theorem}
\begin{proof}
This can be deduced from the long exact sequence and the $\A^1$-invariance of the individual terms.  
The only non-trivial fact is the $\A^1$-invariance of the $F^{p}H^*(-; \C)$ which follows from \cite[Th\'eor\`eme 1.2.10]{hodge2}. 
%For it is well-known that de Rham cohomology is $\A^1$-invariant. Since the projection $\pi$ induces an isomorphism on cohomology and every morphism of smooth complex varieties induces a morphism of Hodge structures, the functor 
%$$F^{p+*}H^*(-; \pi_{2*}E\otimes \C)^n$$ 
%is $\A^1$-invariant too. %by \cite[Th\'eor\`eme 2.3.5]{hodge2}
%Finally, the generalized topological cohomology theory represented by $E$ in the stable homotopy category, $X \mapsto E(X_{\an})$, is $\A^1$-invariant. Propositions \ref{stablecorconstant} and \ref{comparesites} then imply that $Rf_*E$ represents an $\A^1$-invariant theory on the stable homotopy category $\hoSpsPrenis$ as well. 
\end{proof}

\begin{prop}\label{nisdescent}
Let $E$ be a rationally even ring spectrum. The functor 
\[ 
X\mapsto \Elog^*(*)(X)
\] 
on $\Smcnis$ satisfies Nisnevich descent in the sense that a distinguished square in the Nisnevich topology induces a long exact Mayer-Vietoris sequence. 
\end{prop}
\begin{proof}
We have already remarked for the proof of Proposition \ref{compind} that a distinguished square \eqref{dsquare} induces a long exact Mayer-Vietoris sequence for the Hodge filtered cohomology groups $F^p(X;\C)$. 
% presheaf 
%$$H(F^pA^{*}_{\log}(\pi_{2*}E\otimes \C))$$
%of symmetric spectra on $\Smcnis$ 
Since the functor $f^{-1}$ sends distinguished squares to covering squares in the analytic topology, a square of the form \eqref{dsquare} induces a long exact Mayer-Vietoris sequence for $E$-cohomology groups and de Rham cohomology groups too. The exact sequence of Proposition \ref{Elesalg} now shows that $\Elog$ has the desired property as well.
\end{proof}

As a consequence of Theorem \ref{extendeda1} and Proposition \ref{nisdescent}, we get the following result.
\begin{cor}\label{a1functor}
Let $E$ be a rationally even ring spectrum. The functor $\Elog^*(*)(-)$ induces a functor on the $S^1$-stable $\A^1$-homotopy category of smooth varieties over $\C$.
\end{cor}

%%%%%%%%%%%%%%%%%%%%%%%%%%%%%%%%%%%%%
%
\section{Generalized Deligne-Beilinson cobordism for complex algebraic varieties}
In this last section, we apply the refined methods for algebraic complex varieties to the case $E=MU$. The advantage of the logarithmic theory $\MUDB$ is that it has several topological properties that would hold for $\MUD$ only for compact manifolds. In particular, we will show that $\MUDB$ satisfies a projective bundle formula and is equipped with transfer maps for projective morphisms. 

Recall our choice of a map 
$$\bigvee_{p\in \Z} MU \to \bigvee_{p\in \Z} H(\pi_{2*}MU\otimes \C)$$
of commutative $\Sph$-algebras. We can consider this map as a map of presheaves of commutative $\Sph$-algebras on $\Smcnis$. 
Composition with the map  
$$H(\pi_{2*}MU\otimes \C) \to H(A^*(\pi_{2*}MU\otimes \C))$$
and application of $Rf_*$ defines a map of presheaves of commutative $\Sph$-algebras 
$$\bigvee_{p\in \Z} Rf_*MU \to \bigvee_{p\in \Z} Rf_*H(A^*(\pi_{2*}MU\otimes \C)).$$
Moreover, since $\pi_{2*}MU\otimes \C$ is a $\C$-algebra, we obtain a natural analogue of map \eqref{HAlogmap} as a morphism of presheaves of commutative $\Sph$-algebras
\[
\bigvee_{p\in\Z} H(A^{p+*}_{\log}(\pi_{2*}MU\otimes \C))  \to 
\bigvee_{p\in\Z} Rf_*H(A^*(\pi_{2*}MU\otimes \C)).
\]
We define the presheaf of commutative $\Sph$-algebras in $\SpsPrenis$
\begin{equation}
\label{MUlog}
\bigvee_{p\in \Z}\MUlog(p)
\end{equation}
by the homotopy cartesian square 
\begin{equation}
\label{diagramMUDB}
\xymatrix{
*++!{\displaystyle\bigvee_{p\in\Z}\MUlog(p)}  \ar[r]\ar[d]  &  *++!{\displaystyle\bigvee_{p\in\Z} Rf_*MU} \ar[d]\\
*++!{\displaystyle\bigvee_{p\in\Z} H(A^{p+*}_{\log}(\pi_{2*}MU\otimes \C))}  \ar[r]   &
*++!{\displaystyle\bigvee_{p\in\Z} Rf_*H(A^*(\pi_{2*}MU\otimes \C))}.}
\end{equation}

%
%of commutative rings in $\SpsPrenis$.
%
\begin{defn}\label{MUDBrefined}
For a presheaf of symmetric spectra $\Xh$ on $\Smcnis$, we define the logarithmic Hodge filtered complex bordism groups to be 
\[
\MUDB(p)^n(\Xh) := \Hom_{\hoSpsPrenis}(\Xh, \Sigma^n \MUlog(p)).
\]
\end{defn}
Moreover, since \eqref{MUlog} is a presheaf of commutative $\Sph$-algebras, there is a  multiplicative structure on $\MUlog^*(*)$. 
\begin{theorem}\label{DBproduct}
Let $X$ be a smooth complex variety. There is a multiplication   
\[
\MUlog^{n}(p)(X)\otimes \MUlog^{m}(q)(X) \to \MUlog^{n+m}(p+q)(X)
\]
which is graded-commutative in the sense that for $\alpha \in \MUlog^n(p)(X)$, $\beta \in \MUlog^m(q)(X)$ we have
\[
\alpha \beta = (-1)^{n+m}\beta \alpha \in \MUlog^{n+m}(p+q)(X).
\]
This provides 
\[
\MUlog^{*}(*)(X):=\bigoplus_{n,p} \MUlog^{n}(p)(X)
\]
with the structure of a graded-commutative ring. 
%Moreover, if $f \colon \Yh \to \Xh$ is a map of presheaves of symmetric spectra on $\Smcnis$, then $\MUDB^{*}(*)(\Yh)$ is a module over the ring $\MUDB^{*}(*)(\Xh)$ via $f^*$. 
\end{theorem}
%
%
%\begin{prop}\label{lesalgMU}
%Let $X$ be a smooth complex variety. Logarithmic Hodge filtered complex bordism groups sit in a long exact sequence  
%$$\begin{array}{rl}
%\ldots \to MU^{n-1}(X)\otimes \C & \to \MUDB(p)^n(X) \to \\
%\to MU^n(X)\oplus F^{p+*}H^*(X; \pi_{2*}MU\otimes \C)^n & \to MU^n(X)\otimes \C \to \ldots.
%\end{array}$$
%\end{prop}
%\begin{proof}
%This is a consequence of the definition of $MU(p)$ as a homotopy pullback. The only new point is the identification of the Hodge filtered groups given in \eqref{compind}. 
%\end{proof}

\begin{remark}
For a smooth complex variety $X$, the map $MU \to H\Z$ induces a homomorphism 
%$$\MUDB^n(p)(X) \to H_{\Dh\Bh}^n(X;\Z(p))$$
from logarithmic Hodge filtered cobordism to Deligne-Beilinson cohomology which is compatible with the product structures.
\end{remark}

\subsection{Projective bundle formula}

%We recall that by our convention a smooth complex variety is a reduced quasi-projective scheme which is smooth and of finite type over $\C$.

Let $X$ be a smooth variety over $\C$. We would like to define Chern classes 
\[
c_p(V)=c_p^{\MUDB}(V)\in \MUDB^{2p}(p)(X)
\]
for vector bundles $V \to X$ of rank $r$ over $X$. These classes should be functorial with respect to pullbacks and compatible with the Chern classes in the topological theory $MU^{2*}(X)$, i.e. $c_p^{\MUDB}(V)$ should be mapped to the topological Chern classes of $V$ under the map 
\[
\MUDB^{2p}(p)(X) \to MU^{2p}(X).
\]

To define Chern classes we follow the outline of Beilinson in \cite[\S 1.7]{beilinson} for Deligne cohomology together with the general theory of classifying spaces in the homotopy category of simplicial presheaves on $\Smcnis$ in \cite[\S 4.1]{mv}. The following result is a consequence of \cite[Proposition 4.1.15]{mv}.

\begin{prop}\label{bundleclass}
Let $X$ be a smooth complex variety. The set of isomorphism classes of vector bundles of rank $r$ on $X$ is in bijection with the set of maps in the homotopy category of simplicial presheaves from $X$ to $BGL_r$. The correspondence sends a map $f \colon X\to BGL_r$ to the isomorphism class of the bundle $f^*EGL_r$ of the pullback of the universal bundle $EGL_r$ over $BGL_r$. 
\end{prop}
The second ingredient is the following fact that the Hodge filtered cobordism of the classifying space $BGL_r(\C)$ is just given by complex cobordism.
\begin{lemma}\label{BGLlemma}
For every $p\geq 0$, the homomorphism  
\[
\MUDB^{2p}(p)(BGL_r) \to MU^{2p}(BGL_r(\C))
\]
induced by the map of presheaves of symmetric spectra $\MUlog(p) \to MU$ is an isomorphism. 
\end{lemma}
\begin{proof}
This follows from Proposition \ref{Elesalg} and the following two facts. On the one hand, the odd dimensional cohomology of $BGL_r(\C)$ vanishes, i.e.
\[
H^{2p-1+2j}(BGL_r(\C);\pi_{2j}MU\otimes \C) = 0.
\]
On the other hand, the even cohomology of $BGL_r(\C)$ is generated by Chern classes which have pure weight $(p,p)$ (see \cite{hodge3}). This implies 
\[
H^{2p+2j}(BGL_r; \Omega^{* \geq p +j}(\pi_{2j}MU\otimes \C)) \cong H^{2p+2j}(BGL_r(\C); \pi_{2j}MU\otimes \C).
\]
Now one can read off the assertion from the first long exact sequence in Proposition \ref{Elesalg}. 
\end{proof}
%
%The Chern classes of the universal bundle $EGL_r(\C)$over $BGL_r(\C)$ in $\MUD(p)^{2p}(BGL_r(\C))$ would then be given by the image of the Chern classes in cobordism. \\
%
Now let $V$ be a vector bundle of rank $r$ over $X$. By Proposition \ref{bundleclass}, there is a unique map in the homotopy category of simplicial presheaves $f \colon X \to BGL_r$ corresponding to $V$ over $X$. Hence we obtain a pullback map 
\[
\MUDB^*(*)(BGL_r) \to \MUDB^*(*)(X)
\]
corresponding to $V$. The isomorphism of Lemma \ref{BGLlemma} and the Chern classes of the universal bundle over $BGL_r$ in complex cobordism provide Chern classes of the universal bundle in $\bigoplus_p \MUDB^{2p}(p)(BGL_r)$. The above pullback map then defines unique Chern classes $c_p(V)\in \MUDB(p)^{2p}(X)$. This uniqueness implies that these Chern classes satisfy all the properties they satisfy in topological cobordism. 
Moreover, there is a projective bundle formula.

\begin{theorem}\label{pbf}
Let $X$ be a smooth complex variety and $V$ a vector bundle of rank $r$ over $X$ with projective bundle $\Pro(V)\to X$ and tautological quotient line bundle $\Oh_V(1)$. Let $\xi=c_1(\Oh_V(1))$ be the first Chern class of $\Oh_V(1)$ in $\MUDB^2(1)(\Pro(V))$. There are natural isomorphisms 
\[
\bigoplus_{i=0}^{r-1} \xi^{i} \MUDB^{n-2i}(p-i)(X) \to \MUDB^n(p)(\Pro(V))
\]
which make $\MUDB^{*}(*)(\Pro(V)):=\oplus_{p,n} \MUDB^n(p)(\Pro(V))$ into a free $\MUDB^*(*)(X)$-module with basis $1, \xi, \ldots, \xi^{r-1}$.
\end{theorem}
\begin{proof}
This follows from Proposition \ref{Elesalg}. For each of the groups in the long exact sequence we have isomorphisms
\[
MU^{n-1}(\Pro(V)) \cong \bigoplus_{0\leq i \leq r-1} \xi_{MU}^i MU^{n-1-2i}(X),
\]
\[
H^*(\Pro(V);\Vh_{2*})^{n} \cong \bigoplus_{0\leq i \leq r-1} \xi_{H}^i H^*(X; \Vh_{2*})^{n-2i},
\]
\[
F^{p+*}H^*(\Pro(V);\Vh_{2*})^{n} \cong \bigoplus_{0\leq i \leq r-1} \xi_{H}^i F^{p-i+*}H^*(X; \Vh_{2*})^{2p-2i}.
\]
Since the choices of the bases elements are compatible and since the above long exact sequence is compatible with multiplication and direct sums, we deduce that each $\MUDB(p)^n(\Pro(V))$ is a free $\MUDB^{2*}(*)(X)$-module with basis $1, \xi, \ldots, \xi^{r-1}$.
\end{proof}

%%%%%%%%%%%%%%%%%%%%%%%%%%%%%%%%%

\begin{remark}
%The Cartan formula implies the Grothendieck formula 
%By Theorem \ref{pbf}, the elements $\xi^{i}$ form a basis of the $\MUDB^{*}(*)(X)$-module $\MUDB^*(*)(\Pro(V))$. 
Since the Grothendieck formula holds for Chern classes in complex cobordism, Theorem \ref{pbf} and our definition of the $c_p(V)$ imply that the Grothendieck formula
\[
\sum_{p=0}^r (-1)^p c_p(V)\xi^{r-p}=0
\]
also holds for the Chern classes $c_p(V)\in \MUDB^{2p}(p)(X)$. 
\end{remark}

\begin{prop}
Chern classes in Hodge filtered complex cobordism are compatible with Chern classes in complex cobordism and in Deligne cohomology and they are functorial under pullbacks, i.e. for any morphism $f \colon X \to Y$ of smooth projective complex schemes, one has
\[
f^*c_p(V)=c_p(f^*V).
\]
\end{prop}
%\begin{proof}
%The compatibility of the Chern classes with those in complex cobordism and Deligne cohomology follows from the construction and the uniqueness of Chern classes in Deligne cohomology (see \cite[\S 8]{ev}). 
%\end{proof}
%
%%%%%%%%%%%%%%%%%%%%%%%%%%%%%%%%%%
%

\begin{prop}\label{motivicfunctorMU}
The presheaf of symmetric spectra $\MUlog(*)= \bigvee_{p\in \Z} \MUlog(p)$ can be equipped with the structure of a motivic spectrum over $\C$. In particular, the functor $\MUlog^*(*)(-)$ induces a functor on the motivic stable $\A^1$-homotopy category of smooth varieties over $\C$.
\end{prop}
\begin{proof}
Based on Proposition \ref{a1functor}, it remains to show that $\MUlog(*)$ defines a $\Pro^1$-spectrum (not only with respect to $S^1$-suspension). But any choice of lift of the first Chern class 
$c_1(\Oh_{\Pro^1}) \in \MUlog^2(1)$ defines a suspension map 
\[
\Pro^1 \wedge \MUlog(p) \to \MUlog(1) \wedge \MUlog(p) \to \MUlog(p+1)
\]
where the right hand map is given by the ring structure on $\MUlog(*)$. 
%This turns $\MUlog(*)= \bigvee_{p\in \Z} \MUlog(p)$ into a motivic spectrum.
\end{proof}
%
%%%%%%%%%%%%%%%%%%%%%%%%%%%%%%%%%%%%

\subsection{Transfer maps}

Let $f \colon Y\to X$ be a projective morphism between smooth quasi-projective complex varieties of relative codimension $d$. In this section we will define a transfer or pushforward map 
\[
f_* \colon \MUDB^{*}(*)(Y) \to \MUDB^{*+2d}(*+d)(X).
\]
%
%Since $f$ is projective, it admits by definition a factorization $f=q \circ i$ where $i$ is a closed embedding and $q$ is the projection map $q \colon \Pro(V) \to X$ from the projective bundle of some vector bundle $V$ over $X$. Hence in order to define $f_*$, it suffices to define $i_*$ and $q_*$ separately. \\

In order to construct $f_*$, we apply the machinery developed by Panin in \cite{panin1} and \cite{panin2}. For the functor 
\[
X \mapsto \MUDB^*(*)(X):= \bigoplus_{n,p} \MUDB^n(p)(X)
\]
defines an oriented ring cohomology theory on the category of smooth quasi-projective complex varieties in the sense of \cite[\S 2]{panin1}. The existence of an exact localization sequence follows as in Proposition \ref{localization}. The excision property follows from the fact that $\MUDB^n(p)(-)$ satisfies Nisnevich descent as shown in Proposition \ref{nisdescent}. 
The $\A^1$-invariance is the content of Theorem \ref{extendeda1}. This shows that the properties of a cohomology theory in the sense of \cite[Definition 2.1]{panin1} are satisfied. The fact that $\MUDB^*(*)(-)$ is a ring cohomology theory follows from Theorem \ref{DBproduct}. Finally, the projective bundle formula of Theorem \ref{pbf} yields a Chern structure and hence an orientation on $\MUDB$.

Then by \cite[Theorem 2.5]{panin2} and \cite[Theorem 3.35]{panin1}, given a projective morphism $f \colon Y \to X$ between smooth quasi-projective complex varieties of relative codimension $d$, there is a homomorphism 
\[
f_* \colon \MUlog^{*}(*)(Y) \to \MUlog^{*+2d}(*+d)(X).
\]

The construction of $f_*$ is compatible with the Chern structure. In particular, for a smooth divisor 
\[
i \colon D\into X,
\]
the pushforward $i_*$ satisfies  
\[
i_*(1)=c_1(\Oh_X(D)).
\]

Let 
\[
\MUDB^{2*}(*)(X):= \bigoplus_p\MUlog^{2p}(p)(X)
\]
be the sum of the diagonal Hodge filtered cobordism groups. The existence of the transfer structure implies the following result.

\begin{theorem}\label{omegatomudb}
For every smooth quasi-projective complex variety $X$, there is a natural ring homomorphism 
\[
\varphi_{\MUlog} \colon \Omega^*(X) \to \MUlog^{2*}(*)(X)
\]
from algebraic to Hodge filtered cobordism. 
\end{theorem}
\begin{proof}
By \cite{panin2}, the pushforward provides the functor $X\mapsto \MUlog^{2*}(*)(X)$ with the structure of an oriented Borel-Moore cohomology theory on the category of smooth quasi-projective complex varieties in the sense of \cite{lm}. Hence, by \cite{lm}, sending the class $[Y \to X]$ of a projective morphism $f \colon Y\to X$ of relative codimension $d$, with $Y$ smooth and quasi-projective over $\C$, to $f_*(1_Y)\in \MUlog^{2d}(d)(X)$, defines a unique morphism of oriented Borel-Moore cohomology theories
\[
\Omega^*(X) \to \MUlog^{2*}(*)(X).
\]
\end{proof}

\begin{cor}\label{maintheorem}
Let $X$ be a smooth quasi-projective algebraic variety over $\C$. There is a natural ring homomorphism 
\[
CH^* X \xrightarrow{\clMUD} \MUlog^{2*}(*)(X)\otimes_{MU^{\ast}} \Z 
\]
from the Chow ring of algebraic cycles on $X$ modulo rational equivalence such that the composition with the canonical map 
\[
\MUlog^{2*}(*)(X)\otimes_{MU^{\ast}} \Z \xrightarrow{\theta_{\Dh}} H_{\Dh}^{2*}(X;\Z(*))
\]
is the cycle class map $\clHD$ to Deligne-Beilinson cohomology. This map $\clMUD$ is natural with respect to pullbacks and with respect to push-forwards along projective morphisms.  
\end{cor}
\begin{proof}
By \cite[Theorem 1.2.19]{lm}, there is a natural isomorphism  
\[
CH^* \cong \Omega^*\otimes_{\Lee^*}\Z
\]
of oriented cohomology theories on $\Smc$. For a smooth quasi-projective complex variety $X$, it is defined by sending the class of an irreducible subset $Z\subset X$ to the algebraic cobordism class $[\tilde{Z} \to X]$ of a resolution of singularities $\tilde{Z} \to Z$ of $Z$. This map is well-defined and compatible with the additional structures on both sides (see also \cite[\S 4.5]{lm}). The transfer structure on $\MUlog$ and the identification $\Lee^*\cong MU^{2*}$ induce a homomorphism of rings 
%
%%%%%%%%%%%%%%%%%%%%
%
%via diagrams like these
%\[
%\xymatrix@!=15pt{ & \Lee^* \ar[rr] \ar'[d][dd] \ar[dl] & & \Omega^*(X) \ar[dl] \ar[dd] \\ 
%\Z \ar[rr] \ar[dd] & &  \Omega^*(X)\otimes_{\Lee^*}\Z \ar[dd] \\
% & MU^{2*} \ar[dl]\ar'[r][rr] & & \MUlog^{2*}(*)(X) \ar[dl] \\ 
% \Z \ar[rr] & & \MUlog^{2*}(*)(X)\otimes_{MU^*}\Z }
%\]
%%%%%%%%%%%%%%
%\[
%\xymatrix@!=15pt{ & \Lee^* \ar[rr] \ar[dd] \ar[dl] & & \Omega^*(X) \ar[dd] \\ 
%\Z \ar[dd] & &  \\
% & MU^{2*} \ar[dl] \ar[rr] & & \MUlog^{2*}(*)(X) \\ 
% \Z & &  }
%\]
 %%%%%%%%%%%%%%%%%%%
\[
\Omega^*(X)\otimes_{\Lee^*}\Z \to \MUlog^{2*}(*)(X)\otimes_{MU^*}\Z.
\]

Hence we obtain a composed map
\[
CH^*(X) \to \Omega^*(X)\otimes_{\Lee^*}\Z \to \MUlog^{2*}(*)(X)\otimes_{MU^*}\Z.
\]
Since all both maps are ring homomorphisms and compatible with pullbacks and push-forwards along projective morphisms, this proves the theorem.
%
%\[
%\xymatrix{
% & \ar[dl] \Omega^*(X) \ar[d] \ar[r] & \MUlog^{2*}(*)(X) \ar[d] \\
%CH^*(X) \ar[r] & \Omega^*(X)\otimes_{\Lee^*}\Z \ar[r] & \MUlog^{2*}(*)(X)\otimes_{MU^*}\Z.}
%\]
\end{proof}

Let $\Ker^p(\varphi_{MU})(X)$ be the kernel of the natural map $\varphi_{MU} \colon \Omega^p(X) \to MU^{2p}(X)$. Note that the image of $\varphi_{MU}$ is actually contained in the subgroup $\HdgMU^{2p}(X)$. Then the following result is an immediate consequence of the theorem. % we obtain that, for $X$ projective, $\Omega^*(X)$ maps to the fundamental short exact sequence of $\MUlog$. 

\begin{cor}\label{omtoses}
Let $X$ be a smooth projective algebraic variety over $\C$. 
Then there is a functorial commutative diagram of homomorphisms  
\begin{equation}\label{omdiagram}
\xymatrix{
 & \Ker^p(\varphi_{MU})(X) \ar[d]_{AJ_{MU}} \ar[r] & \Omega^p(X) \ar[d]_{\varphi_{\MUlog}} \ar[dr]^{\varphi_{MU}} & & \\
0 \ar[r] & J_{MU}^{2p-1}(X) \ar[r] & \MUlog^{2p}(p)(X) \ar[r] & \HdgMU^{2p}(X) \ar[r] & 0.}
\end{equation}
\end{cor}

\begin{remark}
We think of the natural homomorphism 
\[
AJ_{MU} \colon \Ker^p(\varphi_{MU})(X) \to J_{MU}^{2p-1}(X)
\]
as a generalization of Griffiths' Abel-Jacobi map for algebraic cycles which are homologous to zero. 
\end{remark}

\begin{remark}
The construction of Chern classes and of the transfer map do not depend on the fact that we work with the Thom spectrum $MU$. The required data would be a rationally even ring spectrum $E\in \sSp$ representing a complex oriented cohomology theory, together with a choice of ring map
\[
E\wedge H\C \to H(\pi_{2*}E\otimes \C)
\]
as for the case of $MU$ in \eqref{MUlog}.

Any such data yields an oriented cohomology theory $\Elog^*(*)$ on $\Smc$. It is equipped with transfer maps 
\[
\Elog^*(*)(Y) \to \Elog^{*+2d}(*+d)(X)
\]
for any projective morphism $Y\to X$ of relative dimension $d$ between smooth quasi-projective complex varieties. Moreover, this transfer structure induces a natural map
\[
\Omega^*(X) \to \Elog^{2*}(*)(X)
\]
for every $X \in \Smc$. 

The only reason why we worked out the example of $E=MU$ is that we showed that there exists a choice of multiplicative map 
\[
MU\wedge H\C \to H(\pi_{2*}MU\otimes \C)
\]
and that we are interested in applications based on complex cobordism. But there are of course plenty of other interesting examples.
\end{remark}

\subsection{Examples}

We now offer some simple examples of elements in algebraic cobordism
that map to zero in both the Chow ring and in topological cobordism,
showing that our invariant is finer than the combination of these two.
% \[%\begin{equation}\label{omtoch}
% \varphi_{CH} \colon \Omega^*(X) \to CH^*(X)
% \] %\end{equation}
% and also map to zero under % the map
% \[%\begin{equation}\label{omtomu}
% \varphi_{MU} \colon \Omega^*(X) \to MU^{2*}(X),
% \]%\end{equation}
% but are mapped to non-zero elements under %the map
% \[
% \varphi_{\MUlog} \colon \Omega^*(X) \to \MUlog^{2*}(*)(X).
% \]

Suppose that $X$ is a variety, and $D$ is an element
of $\Omega^{d}(X)$ with the property $\varphi_{MU}(D)=0\in MU^{2d}(X)$.  By the
work of Levine and Morel \cite[Theorem 4.5.1]{lm}, we know that the
kernel of
\[
\varphi_{CH} : \Omega^*(X) \to CH^*(X)
\]
is exactly the ideal
\[
I^*(X):= \Omega^*(X)\cdot \Lee_{*\ge 1}.
\]
So we can get modify $D$ to get an element in the kernel of
$\varphi_{CH}$ by multiplying by, say, the class $\gamma\in\Lee_{1}$
of $\Pro^1 \to \Spec \C$ in $\Omega_1(\C)=\Lee_1$.  To get something
interesting we would like to know that this product is non-zero.

Since $\varphi_{MU}(D)=0$, the cycle class underlying $D$ is
homologous to zero, and so both the Griffiths invariant $AJ(D)\in
J^{2d-1}(X)$ and $AJ_{MU}(D)$ are defined.  If Griffiths' invariant is
not a torsion element, then it is still non-zero in
$J^{2d-1}(X)\otimes \Q$, and so the image of $AJ_{MU}(D)$ in
$J^{2d-1}_{MU}(X)\otimes \Q$ is non-zero.  By
Remark~\ref{remrealliegroup},
\[
\bigoplus_{p}J^{2p-1}_{MU}(X)\otimes \Q
\]
is a flat $\pi_{\ast}MU$-module, and so the class $\gamma\cdot D$ is
also non-zero in $J^{2d+1}_{MU}(X)\otimes \Q$ and so must be non-zero.
Thus we can construct an element $\gamma\cdot D\in \Omega^{d+1}(X)$
with $\varphi_{\MUlog}(\gamma\cdot D)$ non-zero but with
$\varphi_{MU}(\gamma\cdot D)=0$ and $\varphi_{CH}(\gamma\cdot D)=0$ by
finding an element $D\in \Omega^{d}(X)$ with $\varphi_{MU}(D)=0$ and
whose image under Griffiths' Abel-Jacobi map is not a torsion
element.  We describe two examples below.

\begin{example}\label{egpoints}
Let $X$ be a smooth complex curve.  In this case the canonical map $MU^{2}(X)\to
H^{2}(X;\Z)$ is an isomorphism, so any divisor $D$ of degree zero on
$X$ automatically has the property that $\phi_{MU}(D)=0$.  Since the
Abel-Jacobi map from divisors of degree zero to $J^{1}(X)$ is
surjective, when the genus of $X$ is greater than or equal to $1$,
there exist divisors $D$ of degree zero with $AJ(D)$ a non-torsion
element.
\end{example}

A more interesting example is given by one of Griffiths' famous results. 

\begin{example}\label{eggriffiths}
Let $X\subset \Pro^4$ be a general smooth quintic complex
hypersurface.  In \cite[\S\S 13+14]{griffiths} Griffiths showed that
$X$ contains a finite number $n>1$ of lines $L_1, \ldots, L_n$ such
that, for each $i\ne j$, the image $AJ(L_i - L_j)$ under the
Abel-Jacobi map is a non-torsion element in $J^3(X)$ (cf. also
\cite[\S 8.2.3]{voisinbook2}).  By the Lefschetz theorem on hyperplane
sections, the map $X\to \Pro^{4}$ induces an isomorphism of $\pi_{2}$,
so the inclusion of the lines $L_{i}$ and $L_{i}$ into $X$ are
homotopic.  This means that they define the same elements of
$MU_{2}(X)$ and so by Poincar\'e duality, we have that
$\phi_{MU}(L_{i}-L_{j})=0$. 
\end{example}

%
%%%%%%%%%%%%%%%%%%%%%%%%%%%%%%%%%%%%%%%%%
%
%%%%%%%%%%%%%%%%%%%%%%%%%%%%%%%%%%%%%%%%%
%
\bibliographystyle{amsalpha}

\end{document}